\documentclass{amsart}
\usepackage{tikz}
\usetikzlibrary{positioning}
\usetikzlibrary{arrows}
\usepackage{amssymb}
\usepackage{amsfonts}
\usepackage{graphicx}
\usepackage{epstopdf}
\usepackage{algorithmic}
\usepackage{amsmath}
\usepackage{amsthm}
\usepackage{pgfplots}
\usepackage{bigints}
\usepackage{color}
\usepackage{xspace}
\usepackage{array}
\usepackage{mathtools}
\mathtoolsset{centercolon}
\usepackage[capitalise]{cleveref}
\newcommand{\memo}[1]{}
\newcommand{\relmiddle}[1]{\mathrel{}\middle#1\mathrel{}}
\newcommand{\rd}{\mathrm{d}}
\newcommand{\RR}{\mathbb{R}}
\newcommand{\Sb}{\mathbb{S}}

\newcommand{\fd}{\delta^+}

\newcommand{\cd}[1][1]{\delta^{\langle #1 \rangle}}
\newcommand{\adps}{\partial^{\dagger}_x}
\newcommand{\dps}{\delta_{\mathrm{PS}}}

\newcommand{\acd}{\delta^{-1}_{\mathrm{FD}}}

\newcommand{\fa}{\mu^+}

\newcommand{\rank}{\mathop{\mathrm{rank}}\nolimits}
\newcommand{\Dom}{\mathop{\mathrm{dom}}\nolimits}
\newcommand{\Range}{\mathop{\mathrm{range}}\nolimits}
\newcommand{\Null}{\mathop{\mathrm{null}}\nolimits}
\newcommand{\Car}{\mathop{\mathrm{car}}\nolimits}
\newcommand{\im}{\mathrm{i}}
\newcommand{\pinv}[1]{{#1}^{\dagger}}
\newcommand{\ginv}[1]{{#1}^{\mathrm{g}}}
\newcommand{\pid}{\partial_x^{\mathrm{g}}}

\newcolumntype{I}{!{\vrule width 1.5pt}}
\newlength\savedwidth

\newtheorem{theorem}{Theorem}[section]
\newtheorem{lemma}[theorem]{Lemma}

\newtheorem{proposition}[theorem]{Proposition}

\theoremstyle{definition}
\newtheorem{definition}[theorem]{Definition}
\newtheorem{example}[theorem]{Example}

\theoremstyle{remark}
\newtheorem{remark}[theorem]{Remark}

\title[Discretization of evolutionary equations with mixed derivative]{On Spatial Discretization of\\ Evolutionary Differential Equations\\ on the Periodic Domain with a Mixed Derivative}

\author{Shun Sato}
\address{Graduate School of Information Science and Technology, The University of Tokyo, Bunkyo-ku, Tokyo, Japan}
\email{shun\_sato@mist.i.u-tokyo.ac.jp}
\thanks{The first author was supported in part by JSPS Research Fellowship for Young Scientists.}

\author{Takayasu Matsuo}
\address{Graduate School of Information Science and Technology, The University of Tokyo, Bunkyo-ku, Tokyo, Japan}
\email{matsuo@mist.i.u-tokyo.ac.jp}
\thanks{This work was partly supported by JSPS KAKENHI Grant Numbers 25287030, 15H03635, 16KT0016, and 17H02826, and by JST CREST Grant Number JPMJCR14D2, Japan. }

\date{April, 2017}


\begin{document}

\begin{abstract}
Recently, various evolutionary partial differential equations (PDEs) with a mixed derivative have been emerged and drawn much attention. 
Nonetheless, their PDE-theoretical and numerical studies are still in their early stage. 
In this paper, we aim at the unified framework of numerical methods for such PDEs. 
However, due to the presence of the mixed derivative, 
we cannot discuss numerical methods without some appropriate reformulation, which is mathematically challenging itself. 
Therefore, we first propose a novel procedure for the reformulation of target PDEs into a standard form of evolutionary equations. 
This contribution may become an important basis not only of numerical analysis, but also of PDE-theory. 
In order to illustrate this point, we establish the global well-posedness of the sine-Gordon equation. 
After that, we classify and discuss the spatial discretizations based on the proposed reformulation technique. 
As a result, we show the average-difference method is suitable for the discretization of the mixed derivative. 
\end{abstract}

\maketitle

\section{Introduction}
\label{sec_intro}

\subsection{Background}

In this paper, we consider numerical methods for the initial value problem for 
the evolutionary partial differential equations~(PDEs) in the form
\begin{equation}\label{eq_ivp}
\begin{cases}
\left( u_{t} + g (u,u_x,u_{xx},\dots) \right)_x = f(u,u_x,u_{xx},\dots)  & (t \in (0,T), x \in \Sb ),\\
u(0,x) = u_0 (x)  & ( x \in \Sb ),
\end{cases}
\end{equation}
on the periodic domain $ \Sb := \RR / 2 \pi \mathbb{Z}  $. 
Here, $ u : [0,T) \times \Sb \to \RR $ is a dependent variable, $ t $ and $ x $ are temporal and spatial independent variables, subscripts $t$ and $x$ denote the partial derivative with respect to $t$ and $x$, 
and $ u_0 $ is an initial condition. 
Various equations in the form~\eqref{eq_ivp} have been recently emerged and lively studied (see, \cite[Section~3]{SM2017+} for examples and related results on them). 
However, PDE-theoretical and numerical treatments of them are more difficult than usual evolutionary equations 
due to the presence of the spatial differential operator $ \partial_x := \partial / \partial x$ operating on $u_t$. 
We call the resulting term $ u_{tx}$ \emph{mixed derivative} hereafter. 

The spatial differential operator $ \partial_x $ in the mixed derivative is not invertible 
under the standard setting of the space of periodic functions such as the Sobolev spaces 
(the precise meaning will be explained in Section~\ref{sec_pre}). 
Therefore, some problems in the form~\eqref{eq_ivp} are underdetermined while the others are well-posed. 
In this paper, we focus on the latter case. 
Moreover, we restrict ourselves to the spatial discretization (see, Remark~\ref{rem_temp}), 
because the biggest issue is how to treat the spatial differential operator operating on $u_t$. 
In view of this, we mean the operator on $u_t$ when we merely say ``the spatial differential operator,'' 
although other spatial differential operators can appear in the problem~\eqref{eq_ivp} in the functions $g$ and $f$. 

Before stating our challenge in this paper,
we first note that there are existing works for similar initial value problems on the whole real line with the vanishing boundary conditions. 
Since the spatial differential operator $ \partial_x $ is invertible when regarded as a linear operator between 
some appropriate function spaces (see, e.g., I{\'o}rio--Nunes~\cite{IN1998} and references therein), 
the initial value problem can be equivalently transformed into that for the integro-differential equation in the form
\begin{equation*}\label{eq_ivp_real_integral}
u_{t} + g (u,u_x,u_{xx},\dots) = \partial_x^{-1} f(u,u_x,u_{xx},\dots) 
\end{equation*}
(see, e.g., \cite{LPS2009,LPS2010,CR2015} for examples).  
Here, the operator 
\[ \partial_x^{-1} v(x) := \frac{1}{2} \left(  \int^x_{-\infty} v(y) \rd y - \int^{\infty}_x v(y) \rd y \right) \]
is the inverse of the spatial differential operator
(the inverse operator $ \partial_x^{-1} $ is sometimes called as the antiderivative). 
We call the original differential equation \emph{differential form}, 
and the induced integro-differential equation \emph{integral form}. 
Their integral forms have often been utilized rather than differential forms (see, e.g., \cite{LM2006,CR2015})
in order to prove the well-posedness of the equations with a mixed derivative on the whole real line. 

On the other hand, when dealing with numerical methods of evolutionary equations, 
one usually employ a bounded domain and impose some appropriate boundary conditions on it. 
Among several typical boundary conditions, we choose the periodic boundary condition by the following practical reasons: 
it is often used in existing studies on \eqref{eq_ivp} (e.g., \cite{CB2001,LPS2009,LPS2010,YMS2010,MYM2012,CRR2017}); it includes the case of rapidly decreasing functions with sufficiently largely chosen window size, and thus can deal with many practical problems on $ \RR $ with sufficient accuracy; 
it makes numerical analysis simple; it allows the use of pseudospectral method. 
(Hereafter, we focus on the periodic domain unless otherwise stated.)

However, on the periodic domain, 
the spatial differential operator $ \partial_x $ is not invertible as mentioned before. 
This fact makes the reformulation challenging and there is no unified approach so far to derive the integral form. 

Fortunately, however, there are some simple exceptions. 
For example, the reduced Ostrovsky equation
\begin{equation}\label{eq_rOstrovsky}
\left( u_{t} + \left( \frac{1}{2} u^2 \right)_{x} \right)_x = \gamma u. 
\end{equation}
It models water waves on a very shallow rotating fluid, 
and is also referred as the short wave equation~\cite{H1990}, 
Ostrovsky--Hunter equation~\cite{Bo2005}, Vakhnenko equation~\cite{VP1998}, 
and Ostrovsky--Vakhnenko equation~\cite{BS2013}. 

For this equation, Hunter~\cite{H1990} derived the integral form 
\begin{equation}\label{eq_rOstrovsky_integral}
u_t + \left( \frac{1}{2} u^2 \right)_x = \gamma \check{\partial}_x^{-1} u 
\end{equation}
where the operator $ \check{\partial}_x^{-1} $ is defined as 
\begin{equation}\label{eq_anti_linear}
\check{\partial}_x^{-1} v (x) := \int^x_0  v(y) \rd y - \frac{1}{2 \pi} \int_{\Sb} \int^z_0 v(y) \rd y \rd z 
\end{equation}
(the meaning of the unusual notation $ \check{\partial}_x^{-1}$ and the detail of the reformulation will be described in \cref{subsec_der_lin}). 

Based on the integral form~\eqref{eq_rOstrovsky_integral}, 
various studies on the reduced Ostrovsky equation have been conducted. 
Hunter~\cite{H1990} himself conducted numerical experiments based on the integral form 
(however, the discrete counterpart of the operator $ \check{\partial}_x^{-1} $ is not written). 
Liu--Pelinovsky--Sakovich~\cite{LPS2010} showed the local well-posedness and the condition of the finite-time wave breaking, 
and confirmed their theory by numerical experiments using the pseudospectral method. 
Coclite--Ridder--Risebro~\cite{CRR2017} considered a natural numerical scheme based on the integral form~\eqref{eq_rOstrovsky_integral} by using the trapezoidal rule for the discretization of the operator $ \check{\partial}_x^{-1} $, 
and showed its numerical solution converges to the unique entropy solution. 

Similarly, several other equation including the Ostrovsky equation~\cite{O1978,GHO1998} and its generalization~\cite{LL2006} have been studied by using their integral form (see, also \cite{D2017_1,D2017_2} and references therein). 
There, the operator $ \check{\partial}_x^{-1} $ (more precisely its concrete form~\eqref{eq_anti_linear}) 
had been employed, 
but no one has clarified the class of equations to which this strategy can be applicable. 
In fact, equations in the form
\begin{equation}\label{eq_linear}
\left( u_{t} + h_{x} (u,u_x,\dots) \right)_x = u, 
\end{equation}
can be treated in the manner similar to the reduced Ostrovsky equation above, 
and includes many practical equations such as the generalized Ostrovsky equation. 
Note that, due to the absence of physical example in more general form $ (u_t + g(u,u_x,\dots) )_x = u $, 
which can be treated with a slight modification, 
we focus on the form~\eqref{eq_linear} when $ f(u) = u $. 

The important feature of the class of equations in the form~\eqref{eq_linear} is that 
all numerical solutions satisfy the linear implicit constraint $ \int_{\Sb} u(t,x) \rd x = 0 $ 
(this can be verified by integrating both sides of~\eqref{eq_linear} over $x$ on $ \Sb $). 
Since the set of functions satisfying the linear implicit constraint is a linear subspace 
(of a function space), such a reformulation can be successfully conducted (see, \cref{subsec_der_lin}). 
In this sense, we call it \emph{linear case} despite of possible nonlinearity in $h$ (see also Remark~\ref{rem_linear}). 

On the other hand, in general (i.e., $f$ is not linear in \eqref{eq_ivp}), 
the corresponding implicit constraint becomes $ \mathcal{F} (u(t) ) = 0 $, where 
\begin{equation}\label{eq_ic}
\mathcal{F}( v ) = \int_{\Sb} f (v,v_x,\dots) \rd x. 
\end{equation}
Since the set of functions satisfying the constraint above does not form linear subspace in general, 
the reformulation becomes significantly challenging. 
In this sense, we call it \emph{nonlinear case}. 
For example, the well-known sine-Gordon (sG) equation in light cone coordinates
\begin{equation}\label{eq_SG}
u_{tx} = \sin u
\end{equation}
has the nonlinear constraint $ \int_{\Sb} \sin u (t,x) \rd x = 0$, 
and as its consequence, surprisingly even the local well-posedness has not been shown on the periodic domain 
(except for a partial result in \cite{LY2017_2} via the well-posedness result of another equation).

\begin{remark}
The sine-Gordon equation~\eqref{eq_SG} is a special case of the (nonlinear) Klein--Gordon (KG) equation
$ u_{tx} = f(u) $. 
Note that, there are certainly many studies on the KG equation in \emph{Euclidean coordinates} $ u_{\tau \tau} - u_{ss} = f(u)$ (see, e.g., Debnath~\cite{Debnath2012} and references therein), 
and these two representations are related by the simple transformation $ \tau = t+ x, \ s = t- x $ of the independent variables. 
However, the results on the Euclidean case do not give much useful information for the light-cone case. 
In fact, if we consider the initial value problem~\eqref{eq_ivp} for the KG equation in light-cone coordinates, 
the corresponding problem in Euclidean coordinates is the problem 
where the ``initial data'' is given along the line $ \tau = -s $, which seems to be aberrant. 

Still, they have been intensively studied under other boundary conditions. 
Tuckwell~\cite{T2015arx} discussed a finite difference scheme under the initial and boundary conditions
\[ u(0,x) = \alpha (x) \ ( x \in [0,L] ) , \qquad u(t,0) = \beta (t) \ (t \in [0,T) ),  \]
which is introduced by Fokas~\cite{F1997} for the linear KG equation $ u_{tx} = u $ and the sine-Gordon equation~\eqref{eq_SG}. 
Thanks to the presence of the Dirichlet boundary condition, 
such a case is similar to the case on the whole real line. 
Pelloni~\cite{P2005} showed the well-posedness of the sine-Gordon equation with the initial and boundary conditions above. 
Pellinovsky--Sakovich~\cite{PS2010} showed the global well-posedness of the sine-Gordon equation on the whole real line. 
\end{remark}

In addition to the KG equations, 
various new equations with a mixed derivative have been emerged recently: 
mKdV-sG equation~\cite{S2001} (mKdV: modified Korteweg--de Vries), 
the generalized sine-Gordon equation~\cite{F1995},  
the modified Hunter--Saxton equation~\cite{FMN2007} and so on. 
Therefore, their theoretical and numerical studies are indispensable, 
and a unified approach to obtain the corresponding integral form that may become their basis is strongly hoped now. 
However, this has been left open so far, as mentioned before, 
which might be attributed to the fact that 
this becomes surprisingly difficult when $ f $ is nonlinear. 

\subsection{Our contribution}

\subsubsection{Continuous part (\cref{sec_der})}

Below, we address our contributions in the present paper. 
In order to circumvent the difficulty described above, 
we give up to follow Hunter's strategy, 
and propose a novel procedure to derive the integral form in Section~\ref{subsec_der_nl}. 
A key ingredient of the proposed procedure is the Tseng generalized inverse operator $ \ginv{\partial}_x $ (Definition~\ref{def_tgi}) of the differential operator $ \partial_x $, 
which is a standard concept of the generalized inverse for linear operators between Hilbert spaces (see, e.g., \cite{IsraelGreville2003}). 
Proposed procedure can be summarized as follows:
\begin{enumerate}
\item Using the property (Lemma~\ref{lem_inv}) of generalized inverses, 
we obtain $ u_t = g(u,u_x,\dots) + \ginv{\partial}_x f (u,u_x,\dots) + c  $, f
where $ c $ does not depend on $x$. 
\item Under some assumptions, by using the implicit constraint $ \mathcal{F} (u) = 0 $, we determine the value of $c$ by~\eqref{eq_def_const}, and obtain the integral form. 
\end{enumerate}
There, we overcome the difficulty by splitting the derivation into two phases. 
In other words, we use the Tseng generalized inverse in order to tentatively derive the integral form allowing the unknown constant $c$, 
which is then separately determined by using the implicit constraint $ \mathcal{F} (u) = 0 $. 

It should be noted that, the operator $ \check{\partial}_x^{-1} $ defined by~\eqref{eq_anti_linear} can be regarded as a special case of the Tseng generalized inverse, 
so that the reformulation method by Hunter and its followers can be viewed as the special case of our procedure. 
There, the class of equations such that ``Hunter's strategy has been successfully used'' can be interpreted as 
those such that ``$c(t)=0$ thanks to its structure'' (see, Example~\ref{ex_II}). 

Note that, since the spatial differential operator $ \partial_x $ is not invertible, 
the target equations turn out to be infinite-dimensional DAEs (differential-algebraic equations), 
whereas the standard evolutionary PDEs are often regarded as infinite-dimensional ODEs (ordinary differential equations). 
Roughly speaking, proposed procedure can be regarded as an infinite-dimensional extension 
of the geometric reduction~\cite{RR1994} for finite-dimensional DAEs (see, Section~\ref{subsec_geometric}). 
Since the geometric reduction for DAEs is a basis of unique existence theory for nonlinear DAEs, 
and all the well-posedness results for \eqref{eq_ivp} are based on their integral form, 
our contribution may be used for PDE-theoretic studies of the PDEs with a mixed derivative. 
In order to illustrate this point, we establish the global well-posedness of the sine-Gordon equation~\eqref{eq_SG} 
by using its integral form, which is newly derived by the proposed procedure (see, \cref{subsec_wp_sg}). 

\subsubsection{Discrete part (\cref{sec_class})}

Roughly speaking, the proposed procedure reveals the equivalence of the differential and integral forms. 
Therefore, we have two ways to devise spatial discretization, 
i.e., discretizing the differential form or the integral form. 
It should be noted that the equivalence of the two forms is only valid in continuous case; 
the discretization of the two forms are essentially different. 
Section~\cref{sec_class} is devoted to discuss this issue. 

Briefly speaking, 
the discretization of the integral form becomes an ODE, 
while the direct spatial discretization of the differential form is an implicit DAE. 
For such a DAE, by using the discrete analogue of the proposed procedure, 
we can derive a corresponding ODE which can be regarded as a discretization of the integral form. 
There, Tseng generalized inverse of the difference operator can be expressed by  
a generalized matrix of the (matrix expression of) difference operator. 

\begin{remark}
The spatial discretization of the differential form can be written in the simple form 
$D \dot{z} = \phi (z) $, where $ D $ is a singular matrix representing a difference operator and $ \dot{z} = \rd z / \rd t $. 
Its temporal discretization is already discussed in the literature~(see, e.g., Hairer--Wanner~\cite{Hairer-Wanner2010}). 
Moreover, the obtained DAE often has index one (see, e.g., Ascher--Petzold~\cite{Ascher-Petzold1998} for details on index of DAEs), 
and thus, known to be numerically tractable. 
Therefore, in this paper, we focus on the spatial discretization, and do not step into the temporal discretization. 
\label{rem_temp}
\end{remark}

Though most existing numerical methods have been based on the integral form, 
we recommend the \emph{differential form} in practical computation because it
\begin{enumerate}
\item usually has index one; 
\item is free from nonlocal operators (unless $g$ and/or $f$ includes one). 
\end{enumerate} 
On the other hand, discretization of the integral form has a virtue that 
it is fit for some analysis. 

Finally, 
we explore the best spatial discretization that should be employed for the mixed derivative (Section~\ref{sec_mderivative}). 
There, instead of directly analyzing the difference operator itself, 
we investigate its generalized inverse as an approximation of the indefinite integral 
(recall that $ \check{\partial}_x^{-1} $ can be regarded as a generalized inverse of $ \partial_x $, 
which is an indefinite integral).  
In other words, we compare several discretizations of the differential forms by using their integral forms, 
which gives an example of the use of the discretization of the integral form. 
In the present paper, we take this approach, since it seems, at this moment, there is no 
systematic way of evaluating discretization errors when PDE involves the mixed derivative. 

As a result of this exploration, the average-difference method, 
which has been recently introduced by the team including the present authors~\cite{FSM2016}, 
turns out to be superior to other standard methods. 
This fact agrees very well with the numerical observation by Sato--Oguma--Matsuo--Feng~\cite{SOMF2016+} for the sine-Gordon equation. 
Summing up the findings above, we tentatively conclude that, 
for PDEs with a mixed derivative, 
the discretization of the differential form with the average-difference method is recommended. 

%

\subsection{Organization of this paper}

The rest of the paper is organized as follows. 
In Section~\ref{sec_pre}, we show some preliminaries such as function space, variational derivatives, and Tseng generalized inverses. 
The contents in Sections~\ref{sec_der} and \ref{sec_class}  are already described above. 
Then, the paper is concluded in Section~\ref{sec_cr}.

\section{Preliminaries}
\label{sec_pre}

In this paper, $ X^s $ denotes the $s$th Sobolev space on the periodic domain, 
i.e., $ X^s = H^s (\Sb) $ for a nonnegative integer $ s $, 
with the standard inner product. 
Moreover, we define the linear subspace $ \check{X}^s $ of $X^s$ as $\check{X}^s := \left\{ v \in X^s \relmiddle{|} \int_{\Sb} v(x) \rd x  = 0 \right\}$.

\begin{remark}
We believe that our strategy described in Section~\ref{subsec_der_nl} can be extended to other settings 
by appropriately defining the function space (see, Section~\ref{sec_cr}). 
In order to emphasize that, 
we here introduce the symbol $ X^s$. 
\end{remark}

\begin{definition}[Variational derivatives]
For a functional $\mathcal{H} : X^s \to \RR$, its variational derivative $ \delta \mathcal{H} / \delta v (v) $ is defined as a function such that 
\[ \left. \frac{\rd}{\rd \epsilon} \mathcal{H} ( v + \epsilon \phi ) \right|_{\epsilon = 0} = \left\langle \frac{\delta \mathcal{H}}{\delta v} (v) , \phi \right\rangle  \qquad ( \forall \phi \in X^s )\]
holds, where $ \langle \cdot, \cdot \rangle $ is the standard $L^2$ inner product. 
\label{def_vd}
\end{definition}

When the funcitonal $ \mathcal{H} $ is defined by $ \mathcal{H} : v \mapsto \int_{\Sb} G (v,v^{(1)},\dots,v^{(k)}) \rd x$,
its variaitonal derivative can be calculated by
\begin{equation*}
\frac{\delta \mathcal{H}}{\delta v} (v) = \sum_{i=0}^k ( - \partial_x )^{i} \frac{\partial G}{\partial v^{(i)}} (v,v^{(1)},\dots,v^{(k)}) ,
\end{equation*}
where $ v^{(i)} $ is the $i$th derivative of $v$ for $i \ge 1$ and $ v^{(0)} = v $. 
We often use the abbreviation $ \delta \mathcal{H} / \delta v $ for simplicity. 

Furthermore, we introduce the generalized inverse of a linear operator between Hilbert spaces (see, e.g., \cite{IsraelGreville2003}). 
Here, for a linear operator $L : Y_1 \to Y_2 $ between two Hilbert spaces $Y_1, Y_2$, 
$ \Dom (L) \subseteq Y_1 $ and $ \Range (L) \subseteq Y_2$ denote the domain and range of $L$, respectively. 
For a closed subspace $A$ of a Hilbert space $ Y_1 $, $ P_{A} :Y_1 \to Y_1$ denotes the orthogonal projector on $ A $ (i.e.,  $ \Range (P_A) = A $ holds).

\begin{definition}[Tseng generalized inverses] \label{def_tgi}
Let $ L :Y_1 \to Y_2 $ be a linear operator. 
Then a linear operator $ \ginv{L} : Y_2 \to Y_1 $ is a \emph{Tseng generalized inverse} of $L$ if it satisfies the following four conditions: 
\begin{align*}
&\Range (L) \subseteq \Dom (\ginv{L}) ,&
&\Range (\ginv{L}) \subseteq \Dom (L) ,\\
&\ginv{L} L x = P_{\overline{\Range(\ginv{L})}} x \quad ( x \in \Dom (L) ),&
&L \ginv{L} x = P_{\overline{\Range(L)}} y \quad  (y \in \Dom (\ginv{L}) ).
\end{align*}
\end{definition}

Moreover, we introduce the null space $ \Null (L) $ and the career $ \Car (L) $ of $L$, i.e., $ \Null (L) = \{ x \in \Dom (L) \mid L x =0 \} $, $  \Car (L) = \Dom (L) \cap \Null (L)^{\perp} $, 
where $ A^{\perp} $ stands for the orthogonal complement of a linear subspace $A$. 

\begin{lemma}[\protect{\cite[Chapter 9, Lemma~3]{IsraelGreville2003}}]
If $ \ginv{L} $ is a Tseng generalized inverse of $L$, then  
$ \Null (L) = \Dom (L) \cap \Range (\ginv{L})^{\perp}$ and $ \Car (L) = \Range (\ginv{L})$ hold. 
\label{lem_tgi}
\end{lemma}

We also define the maximal generalized inverse operator as follows. 

\begin{definition}[Maximal Tseng generalized inverse]
Let $L:Y_1 \to Y_2$ be a linear operator. 
Then, a linear operator $ \pinv{L} : Y_2 \to Y_1 $ is called a \emph{maximal Tseng generalized inverse operator} if 
it is a Tseng generalized inverse operator satisfying $ \Dom ( \pinv{L} ) = Y_2 $. 
\label{def_mtgi}
\end{definition}

Note that $ \ginv{L} y = \pinv{L} y $ holds for any $ y \in \Range(L)$. 
This can be checked by using $ x \in \Dom (L) $ such that $ y = Lx $ as follows: 
$ \ginv{L} y = \ginv{L} L x = \pinv{L} L x = \pinv{L} y$. 
We will use this property in Section~\ref{sec_mderivative}. 

Now, let us introduce Tseng generalized inverse operators of the spatial differential operator 
$ \partial_x : X^s \to X^{s-1} $ (see, e.g., \cite[Example 1]{H1961}). 
Here, we assume $ s \ge 1 $.  
Note that, the differential operator satisfies 
$ \Dom (\partial_x) = X^s $, $ \Range (\partial_x ) = \check{X}^{s-1} $, $ \Car (\partial_x) = \check{X}^s $, and 
$ \Null (\partial_x ) = \{ \alpha \mathbf{1} \in X^s \mid \alpha \in \RR \} $. 
Here, $ \mathbf{1} $ denotes a constant function satisfying $ \mathbf{1} (x) = 1 \ (x \in \Sb)$.
Thus, by Definition~\ref{def_tgi} and Lemma~\ref{lem_tgi}, each generalized inverse operator $ \pid $ satisfies
\begin{align*}
&\Range ( \partial_x ) \subseteq \Dom ( \pid ), &
&\Range ( \pid) = \Car(\partial_x) = \check{X}^s, \\
&\pid \partial_x v = P_{\check{X}^s} v \quad (v \in X^s), &
&\partial_x \pid w = P_{\check{X}^{s-1}} w \quad (w \in \Dom ( \pid )).
\end{align*}

It should be noted that, since $ \Dom (\pid) $ should be a linear subspace of $ X^{s-1} $ and $ \dim \Range ( \partial_x )^{\perp} = 1 $, 
$ \Dom (\pid) = \check{X}^{s-1} $ or $ \Dom (\pid) = X^{s-1} $ hold. 
This fact implies that there are only two Tseng generalized inverse operators, 
because the Tseng generalized inverse is uniquely determined by its domain. 

The Tseng generalized inverse $\pid$ satisfying $ \Dom (\pid) = \Range (\partial_x) = \check{X}^{s-1} $
 can be concretely expressed as 
\[ (\pid v ) ( x) := \int^x_0 v(y) \rd y - \frac{1}{2 \pi} \int_{\Sb} \int^z_0 v(y) \rd y \rd z \qquad ( v \in \check{X}^{s-1} ), \] 
which coincides with $ \check{\partial}_x^{-1} $ (see, \eqref{eq_anti_linear}) introduced by Hunter~\cite{H1990}. 
Since $ \int^x_0 v(y) \rd y $ is not periodic if $ v $ is not zero-mean, 
the definition above ceases to work for $ v \in X^{s-1} \setminus \check{X}^{s-1} $ 
so that this operator is not maximal. 

On the other hand, the maximal Tseng generalized inverse $\adps$ can be expressed by using the Fourier series as follows: 
\begin{align*}\label{eq_inv_linear_PS}
\adps v(x) &:= \sum_{- \infty < k < \infty, \, k \neq 0 } \frac{\hat{v}(k)}{k \im} \exp \left(k \im x \right),&
\hat{v} (k) &:= \frac{1}{ 2 \pi} \int_{\Sb} v(x) \exp \left( - k \im x \right) \rd x, 
\end{align*}
where $ \im $ is the imaginary unit. 
In fact, this operator was already introduced by Yaguchi--Matsuo--Sugihara~\cite{YMS2010} as 
an alternative to $ \check{\partial}_x^{-1} $. 
They used $ \adps $ in order to describe the pseudospectral method for the Ostrovsky equaiton~\eqref{eq_Ostrovsky} as a discretization of it. 

In what follows, we use the symbol $ \check{\partial}_x^{-1} $ and $ \adps $ when we need to indicate each of the specific Tseng generalized inverses, 
while $ \pid $ is employed when we allow both. 

\begin{remark}
Although the concrete examples of Tseng generalized inverses of the spatial differential operator 
had been used for various equations in the literature, 
no one has explicitly described that they can be regarded as a Tseng generalized inverse, to the best of the present authors' knowledge. 
\end{remark}

The following lemma is an immediate corollary of Lemma~\ref{lem_tgi}: 

\begin{lemma}
For any $ v \in X^s $ and any Tseng generalized inverse operator $ \pid $ of $ \partial_x $, 
there exists a constant $ c \in \RR $ such that
$v (x)= \pid \partial_x v (x) + c $ holds for any $ x \in \Sb$.
\label{lem_inv}
\end{lemma}

\begin{proof}
By Lemma~\ref{lem_tgi} and Definition~\ref{def_tgi}, for any $ v \in X^s $ it holds that 
\[ (\pid \partial_x  v) (x) = \left( P_{\overline{\Range ( \pid )}} v \right) (x) = \left(  P_{\check{X}^s} v \right) (x) = v (x) - c \quad  (  x \in \Sb  )\]
for some constant $ c \in \RR $, 
which proves the lemma.  
\end{proof}

\section{Derivation of integral forms}
\label{sec_der}

In Section~\ref{subsec_der_lin}, 
we review the derivation of the integral form of linear case~\eqref{eq_linear}, 
which has been already known in the literature. 
Then, we propose a new procedure for the nonlinear case in Section~\ref{subsec_der_nl}. 
Finally, Section~\ref{subsec_wp_sg} is devoted to an example of applications of the proposed procedure: 
we prove the global well-posedness of the sine-Gordon equation~\eqref{eq_SG}.

\subsection{Linear case}
\label{subsec_der_lin}

Here, we show the derivation of integral forms of equations in the form~\eqref{eq_linear}. 
This can be done in the manner similar to the case of the reduced Ostrovsky equation~\eqref{eq_rOstrovsky}. 
However, since there is no explicit explanation in Hunter~\cite{H1990} and its followers, 
we here show our understanding. 

For simplicity, we use the abbreviation such as $ h(u) $ which stands for $ h(u,u_x,\dots) $ in what follows. 
In this section, let us assume the initial condition $ u_0 $ satisfies $ \int_{\Sb} u_0 (x) \rd x = 0 $, the map $ h $ is $ h : X^{s+k} \to X^{s+1} $ for some positive integer $k$, and $ s $ is a positive integer. 

\begin{remark}
	As said in \cref{sec_intro,sec_pre}, $ \check{\partial}_x^{-1} $ can be regarded as a special case of Tseng generalized inverses of the differential operator. 
	However, this fact has not been stated in the literature. 
	Thus, we forget this fact at the moment, and show how $ \check{\partial}_x^{-1} $ naturally appears. 
	At the same time, we describe the meaning of the unusual notation. 
	That is, as we show below, the inverse operator $ \check{\partial}_x $ of a restricted differential operator. 
	(The derivation below will be rephrased by using the concept of Tseng generalized inverse later on: Example~\ref{ex_II}.)
\end{remark}

Let us assume that $ u \in C ([0,T);X^{s+k} (\Sb)) \cap C^1 ([0,T);X^s) $ is a solution of the differential form~\eqref{eq_linear} satisfying the initial condition (cf. \cite[Lemma~1]{LPS2010} for the reduced Ostrovsky equation).  
Then, as mentioned in \cref{sec_intro},  $ \int_{\Sb} u(t,x) \rd x = 0 $ is satisfied. 
Hereafter, for $ t \in [0,T)  $, we use the notation $ u(t) $ denoting the element in $ X^{s+k}  $ such that $ ( u(t) ) (x) = u(t,x) $ holds for any $ x \in \Sb $. 

This linear constraint implies $ \int_{\Sb} u_t (t,x) \rd x = 0 $ so that $ u_t (t) + h_x (u(t)) \in \check{X}^s $ holds
(recall that $ \check{X}^s = \{ v \in X^s \mid \int_{\Sb} v (x) \rd x = 0 \} $). 
Thus, the spatial differential operator $ \partial_x $ operating on $ u_t (t) + h_x (u(t)) $ can be replaced by the 
restricted operator $ \check{\partial}_x := \partial_x | \check{X}^s $ (i.e., $ \check{\partial}_x:  \check{X}^s \to X^{s-1} $). 
In other words, equation~\eqref{eq_linear} can be rewritten in the form 
\[ \check{\partial}_x  \left( u_t (t) + h_x (u(t)) \right) = u. \]
Since the restricted operator $ \check{ \partial }_x $ is injective, 
there exists the inverse operator $ \left( \check{\partial}_x \right)^{-1} : \check{X}^{s-1} \to \check{X}^s $, 
which coincides with the operator $ \check{\partial}_x^{-1} $. 
Summing up,  we obtain the integral form
\begin{equation}\label{eq_linear_int}
u_t (t) + h_x (u(t)) = \check{\partial}_x^{-1} u (t). 
\end{equation}

The key of the reformulation above is the exquisite linear subspace $ \check{X}^s $. 
This linear subspace satisfies
\begin{enumerate}
	\item[(a)] $ u_t (t) + h_x (u(t)) \in \check{X}^s $ holds so that the restricted operator $ \check{\partial}_x $ can be used; 
	\item[(b)] the restricted operator $ \check{\partial}_x $ is injective so that its inverse exists. 
\end{enumerate}

\begin{remark}\label{rem_linear}
	In Introduction, we describe the unusual terminology \emph{linear case} is due to the linear implicit constraint. 
	However, in fact, this terminology has stronger meaning as shown above: 
	in linear case, all discussions can be done in the linear subspace $ \check{X}^s $. 
\end{remark}

\subsection{New result: nonlinear case}
\label{subsec_der_nl}

Our purpose here is to devise the procedure to derive the integral form for general case~\eqref{eq_ivp}. 
However, if we try to extend the strategy in the previous section,  
we should consider ``tangent space of the Hilbert manifold $ \{ v \in X^s \mid \mathcal{F} (v) = 0 \} $ at $ u(t) $'' in place of $ \check{X}^s $ (see, e.g., \cite{AMR1988} for details on infinite-dimensional manifold), and discuss the invertibility of the restriction of the differential operator on it. 
This is not impossible (cf. \cite{SatoMT}), but mathematically too complicated. 
Therefore, in \cref{subsub_proc}, we show a simple and effective approach by using the concept of the generalized inverse. 
After that, we illustrate the geometric interpretation of the proposed procedure in \cref{subsec_geometric}.

\subsubsection{Proposed procedure}
\label{subsub_proc}

We suppose $ \mathcal{F} (u_0) = \int_{\Sb} f(u_0) (x) \rd x = 0 $, $ s $ is a positive integer, and 
there exists a variational derivative $ \delta \mathcal{F} / \delta u $ of $ \mathcal{F} $ in this section. 
Moreover, we assume there exists a solution $ u \in C^0 ([0,T]; X^{s+k}) \cap C^1 ( [0,T]; X^{s} ) $ of the initial value problem~\eqref{eq_ivp}, 
where $k$ is a nonnegative integer such that 
$ f: X^{s+k} \to X^{s-1} $ and $ g : X^{s+k} \to X^s $. 

\begin{remark}
Note that, at the moment, we do not necessarily know whether the assumption on existence of a solution are satisfied. 
As we will illustrate in \cref{subsec_wp_sg}, 
the well-posedness of the original differential form should be discussed 
after the well-posedness of the derived integral form is successfully established. 
\end{remark}

By operating on the both sides of
\begin{equation}\label{eq_nl}
\partial_x \left( u_t (t) + g(u(t) ) \right) = f(u(t))
\end{equation}
with a Tseng generalized inverse operator $ \pid : X^{s-1} \to X^s$, 
for any $ t \in (0,T) $, we see
\begin{equation}
u_t (t) + g (u(t)) = \pid f(u(t)) + c(t) \mathbf{1},
\label{eq_inter}
\end{equation}
where $ c(t) $ does not depend on $x$ from Lemma~\ref{lem_inv}. 
In some happy cases, the implicit constraint $ \mathcal{F}(u(t)) = 0 $ enables us to determine the value of $ c(t)$ as follows. 

Since the value of $ \mathcal{F} (u(t)) $ is always $0$ and $ u(t) $ satisfies 
\begin{equation}\label{eq_key}
\frac{\rd}{\rd t} \mathcal{F} (u(t)) 
= \left\langle \frac{\delta \mathcal{F}}{\delta u } , u_t \right\rangle
= \left\langle \frac{\delta \mathcal{F}}{\delta u} , - g(u) + \pid f(u) \right\rangle + c(t) \left\langle \frac{\delta \mathcal{F}}{\delta u} , \mathbf{1} \right\rangle,
\end{equation}
the value of $c(t) $ is determined 
for $ t \in (0,T)$ satisfying  
\begin{equation}
\left\langle \frac{\delta \mathcal{F}}{\delta u} , \mathbf{1} \right\rangle  \neq 0 . 
\label{eq_condition}
\end{equation}
In this case, we obtain the desired integral form
\begin{equation}\label{eq_nl_int}
u_t (t) + g(u(t)) = \pid f(u(t)) + \mathcal{C} (u(t)),
\end{equation}
where
\begin{equation}\label{eq_def_const}
\mathcal{C} (v) := \frac{ \left\langle \frac{\delta \mathcal{F}}{\delta u}  , g(u) - \pid f(u) \right\rangle }{ \left\langle \frac{\delta \mathcal{F}}{\delta u} , \mathbf{1} \right\rangle}. 
\end{equation}

By construction, the following proposition holds. 
\begin{proposition}\label{prop_ic}
	For any solution $u$ of~\eqref{eq_nl_int} and $ t \in (0,T) $ satisfying~\eqref{eq_condition}, $ (\rd / \rd t) \mathcal{F} (u(t)) = 0 $ holds. 
	In particular, if $ \mathcal{F} (u_0) = 0 $ holds and \eqref{eq_condition} holds for any $ t \in (0,T) $, 
	then $ \mathcal{F} (u(t)) = 0 $ holds for any $ t \in [0,T) $. 
\end{proposition}

On the other hand, for $ t \in (0,T )$  such that $ \langle \delta \mathcal{F} / \delta u , \mathbf{1} \rangle = 0 $ holds, we obtain a new implicit constraint $ \mathcal{F}_1 (u(t)) = 0 $, where 
\begin{equation}
\mathcal{F}_1 (v)  := \left\langle \frac{\delta \mathcal{F}}{\delta v} , g(v) - \pid f(v) \right\rangle 
\label{eq_ic1}
\end{equation}
from~\eqref{eq_key}. 
When there exists an open interval $ (t_0,t_1) \subseteq (0,T) $ such that 
$ \langle \delta \mathcal{F} / \delta u , \mathbf{1} \rangle = 0 $ holds for any $ t \in (t_0,t_1) $, we can continue the same line of discussion above by using $ \mathcal{F}_1 $ instead of $ \mathcal{F} $. 
Otherwise, $ c(t) $ will be determined by continuity. 
As a result, the conservation law such as Proposition~\ref{prop_ic} holds even for this case, 
which can be shown in the similar manner.  
However, since the statement is quite cumbersome and there is no physical example in this case as we show below, 
the detailed discussion is omitted. 


Now, in the example below, 
we show how the derivation of the integral form for linear cases~\eqref{eq_linear} in the previous section can be understood in the proposed procedure. 

\begin{example}[Linear case~\eqref{eq_linear} revisited]\label{ex_II}
The implicit constraint 
$ \mathcal{F} (u) = \int_{\Sb} u \, \rd x =  0 $
is linear, and it holds that $ \langle \delta \mathcal{F} / \delta u , \mathbf{1} \rangle = 2 \pi \neq 0$. 
Therefore, we see
\begin{equation*}
\mathcal{C} (v) = \frac{1}{2\pi} \left( \langle \mathbf{1} ,  h_x (v) \rangle - \langle \mathbf{1} ,  \pid v \rangle \right) = 0, 
\end{equation*}
where the last equality holds by $ \Range (\pid) = \check{X}^s $. 
Thus, it coincides with the transformation by Hunter~\cite{H1990} and its followers if we employ $ \check{\partial}_x^{-1} $ as the generalized inverse $ \pid $.
\end{example} 


For the nonlinear case, 
whether the condition~\eqref{eq_condition} holds or not 
depends on each solution in general. 

\begin{example}
If we consider the nonlinear Klein--Gordon equation of the form
\[ u_{tx} = u + u^2, \]
we obtain $ \delta \mathcal{F} / \delta u = \mathbf{1} + 2 u $. 
Then, the value of $ \langle \mathbf{1},  \mathbf{1} + 2 u \rangle $ depends on time. 
\end{example}

Fortunately, however, physically meaningful equations somehow tend to satisfy the condition~\eqref{eq_condition} for any $ t \in (0,T) $ thanks to some associated conservation laws 
as far as the initial condition $u_0$ satisfies $ \langle  \delta \mathcal{F} / \delta u ( u_0 ) , \mathbf{1} \rangle \neq 0 $.  
For example, the sine-Gordon equation (Section~\ref{subsec_wp_sg}), the modified Hunter--Saxton equation (Example~\ref{ex_mHS}), and the modified short pulse equation below. 


\begin{example}
The modified short pulse equation 
\begin{equation}\label{eq_mSP}
	u_{tx} = u + \frac{1}{2} u \left( u^2 \right)_{xx}. 
\end{equation}
has the nonlinear implicit constraint 
\[ \mathcal{F} (u) = \int_{\Sb} \left( u + \frac{1}{2} u \left(u^2 \right)_{xx} \right) \rd x = \int_{\Sb} \left( u - u u_x^2 \right) \rd x.  \]
But in this case, since 
\[ \frac{\rd}{\rd t} \int_{\Sb} \frac{1}{2} u_x^2 \rd x = \int_{\Sb} u_x u_{xt} \rd x = \int_{\Sb} u_x \left( u + \frac{1}{2} u \left(u^2 \right)_{xx} \right) \rd x = 0 \]
and 
\begin{equation*}
\left\langle \frac{\delta \mathcal{F}}{\delta u } , \mathbf{1} \right\rangle = \left\langle 1 - u_x^2 + \left( 2 u u_x \right)_x , \mathbf{1} \right\rangle
= 2 \pi - \int_{\Sb} u_x^2 \rd x, 
\end{equation*}
the condition~\eqref{eq_condition} holds for any $ t \in (0,T) $ if $ 2 \pi \neq \int_{\Sb}  \left( (u_0)_x \right)^2 \rd x $. 
\label{ex_mSP}
\end{example}

\begin{remark}
	After our preprint~\cite{SM2017+} had been published, 
	Li--Yin~\cite{LY2017_2} showed the local well-posedness and global existence of the modified short pulse equation 
	via the reformulation, which essentially coincides with the integral form derived by our procedure. 
	There, though it seems that the similar transformation was conducted, 
	how to determine the constant $ c(t) $ (they call it ``boundary term'') is unclear and just said ``carefully selected''
	(see, also \cite{LY2017_1} for the similar work for the Hunter--Saxton-type equation). 
	In view of this, their work indicates that the proposed procedure is an important tool for this kind of PDEs. 
\end{remark}

%

Next, let us consider what happens when we cannot ensure $ \langle \delta \mathcal{F} / \delta u , \mathbf{1} \rangle \neq 0 $ 
through the geometric interpretation.

\subsubsection{Geometric Interpretation}
\label{subsec_geometric}

This section is devoted to illustrate the intuition of the procedure in the previous section. 
To this end, we consider its geometric interpretation. 
Roughly speaking, the proposed procedure can be viewed as the infinite-dimensional version of the 
``geometric reduction'' for finite-dimensional implicit DAEs~\cite{R1990,RR1994}, 
which gives us the local existence and uniqueness results of them. 

However, since the general well-posedness theory of PDEs~\eqref{eq_ivp} is beyond the scope of this paper, 
we do not step into the rigorous justification of the infinite-dimensional version of the geometric reduction. 
Such a justification seems to be challenging,
because, for example, the definition of the reducibility itself (see, \cite[Definition~4.2]{RR1994}) is only valid for finite-dimensional cases. 

First of all, along~\cite{RR1994}, 
we consider a reduction process for general PDAEs~(partial differential-algebraic equations) in the form 
\begin{equation}
F(u,u_t) = 0, 
\label{eq_pdae}
\end{equation}
where $ F : X^s \times X^s \to X^s $ is an arbitrary smooth map. 
When we define the Hilbert manifold $ M = F^{-1} (0) $,
the equation~\eqref{eq_pdae} is equivalent to 
\begin{equation}
(u,u_t) \in M.
\label{eq_pdae_manifold}
\end{equation}
Here, we regard $M$ as a submanifold of $ T X^s $, where $ T X^s$ stands for the tangent bundle of $X^s$. 
Suppose that the canonical projection $ W = \pi (M) $ is a submanifold of $X^s$ ($ \pi : T X^s \to X^s $ is a map such that $ \pi : (u,v) \mapsto u $). 
By definition of $W$, the solution $ u $ of the equation~\eqref{eq_pdae} satisfies $ (u,u_t ) \in TW $. 
Then, the solution $u$ of the equation~\eqref{eq_pdae} also satisfies
\begin{equation}
(u,u_t ) \in M_1 := M \cap TW. 
\end{equation}
The process obtaining $M_1$ from $M$ is called the geometric reduction, 
and we can further proceed the reduction step such as $ W_{i+1} = \pi (M_i) $, $ M_{i+1} = T W_{i+1} \cap M_i $  ($ i = 1,2,\dots$). 

Now, let us restrict ourselves to the case $F(u,v) := v_x + g_x (u) - f(u)$, i.e., the equation~\eqref{eq_nl}. 
In this case, $M$ can be explicitly written in the form 
\[ M = \left\{ (u,v) \relmiddle{|} \mathcal{F} (u) = 0, \ v = - g(u) + \pid f(u) + c \mathbf{1}  \ ( c \in \RR )  \right\} \]
(recall~\eqref{eq_inter}).
Then, the canonical projection $W$ of $M$ and its tangent bundle $ TW $ can be constructed by 
\begin{align*}
W &= \pi (M) = \left\{ u \in X \mid \mathcal{F}(u) = 0 \right\}, \\
TW &= \left\{ (u,v) \relmiddle{|} \mathcal{F} (u) = 0, \ \left. \frac{\rd}{\rd \epsilon} \mathcal{F} (u+ \epsilon v) \right|_{\epsilon = 0} = 0 \right\}.
\end{align*}
In a manner similar to that in the previous section, 
$ M_1 $ can be written in the form
\begin{equation}\label{eq_gr}
\begin{split}
M_1 &= \left\{ (u,v) \in M \relmiddle{|} \left\langle \frac{\delta \mathcal{F}}{\delta u } , \mathbf{1} \right\rangle \neq 0 , \ v = - g(u) + \pid f(u) + \mathcal{C} (u) \mathbf{1} \right\} \\
& \qquad \cup \left\{ (u,v) \in M \relmiddle{|} \left\langle \frac{\delta \mathcal{F}}{\delta u } , \mathbf{1} \right\rangle = 0 , \ \mathcal{F}_1 (u) = 0 \right\}. 
\end{split}
\end{equation}
Here, for $ u \in W $ such that $ \langle \delta \mathcal{F}/ \delta u , \mathbf{1} \rangle \neq 0 $, 
the associated tangent vector $v$ is uniquely determined 
(namely, the constant $c$ is determined by $\mathcal{C}(u)$, which is defined by~\eqref{eq_def_const}), 
while we obtain a new constraint $ \mathcal{F}_1 (u)=0 $ for $u \in W$ such that $ \langle \delta \mathcal{F}/ \delta u , \mathbf{1} \rangle = 0 $. 
At every step $i$, we obtain $ M_i $ by a similar reduction process. 

When we assume the reduction procedure is successfully well-defined, i.e., 
the sequence of the Hilbert manifolds $ \{ M_i \}_{i=1}^{\infty} $ can be defined, 
there are three possible scenarios as follows: 
\begin{enumerate}
\item[(a)] $ M_i = M_{i+1}$ for some finite positive integer $i$, and there is \emph{exactly one} tangent vector $v$ such that $ (u,v) \in M_i $ for any $ u \in \pi (M_i) $: \\
In this case, the infinite-dimensional vector field can be uniquely determined. 
In other words, the equation can be rewritten in the integral form for any initial conditions $ u_0 \in \pi (M_i) $. 
\item[(b)] $ M_i = M_{i+1}$ for some finite positive integer $i$, but this time there are \emph{more than one} tangent vectors $v$ such that $ (u,v) \in M_i $ for some $ u \in \pi (M_i) $: \\
In this case, the result of the reduction process does not provide the integral form at least for some initial conditions. 
\item[(c)] $ M_i \neq M_{i+1} $ holds for any positive integer $i$: \\
This case is the distinctive scenario of the infinite-dimensional case (see, example~\ref{ex_infinite} below). 
We cannot obtain the integral form by the reduction process above. 
\end{enumerate}

Below, we show an example which has the infinitely many implicit constraints, i.e., the case (c). 

\begin{example}
We consider the following PDE
\begin{equation}
u_{tx} = \frac{1}{3} u_x^3, 
\label{eq_inf}
\end{equation}
which is obviously underdetermined since if $ u (t,x) $ is a solution, then $ u (t,x) + d (t) $ is also a solution for any $ d : [0,T) \to \RR $ satisfying $ d(0) = 0 $. 
Since 
$ \langle \delta \mathcal{F} / \delta u , \mathbf{1} \rangle  =  \langle - (  u_x^2 )_x ,\mathbf{1} \rangle = 0 $ holds, 
we then proceed to obtain a new implicit constraint
\begin{align*}
\mathcal{F}_1 (u) &=  - \left\langle  \left( u_x^2 \right)_x , - \pid \left( \frac{1}{3} u_x^3 \right) \right\rangle = - \frac{1}{3} \left\langle u_x^2 , \partial_x \pid u_x^3 \right\rangle
= - \frac{1}{3} \int_{\Sb} u_x^5 \rd x.
\end{align*}
In this manner, we can repeat this procedure, and 
it is easy to verify that at every step $i$ ($ i=1,2,\dots $),  $ \mathcal{F}_i = \alpha_i \int_{\Sb}  u_x^{2i+3} \rd x $ holds for some constant $ \alpha_i$. 
In other words, $ M_i \neq M_{i+1} $ holds for any positive integer $i$. 
\label{ex_infinite}
\end{example}

Note that, when $ \langle \delta \mathcal{F} / \delta u , \mathbf{1} \rangle \neq 0 $ holds for any $ u \in W $, 
$M_1$ (defined by~\eqref{eq_gr}) can be simply expressed as 
\[ M_1 = \left\{ (u,v) \relmiddle{|}  \mathcal{F} (u) = 0,  v = - g(u) + \pid f(u) + \mathcal{C} (u) \mathbf{1}  \right\},  \] 
and $ M_i = M_1 $ holds for any positive integer $ i $. 
It should be noted that, the discussion on the linear case in \cref{subsec_der_lin} can be rephrased as this case. 
However, it is not the case for the equation~\eqref{eq_nl} in general. 
Still, the modified Hunter--Saxton equation belongs to this case. 

\begin{example}\label{ex_mHS}
The modified Hunter--Saxton equation~\cite{FMN2007}
\begin{equation}\label{eq_mHS}
\left( u_{t} + \frac{1}{2} \left( u^2 \right)_{x} + \frac{\gamma}{6} u_x^3 \right)_x = u + \frac{1}{2} u_x^2,
\end{equation}
describes the propagation of short waves in a long wave model. 
For this case, the implicit constraint can be expressed as $ \mathcal{F} (u) = \int_{\Sb} \left( u + \frac{1}{2} u_x^2  \right) \rd x = 0 $. 
Since
\[ \left\langle \frac{\delta \mathcal{F}}{\delta u} , \mathbf{1} \right\rangle = \left\langle 1 - u_{xx} , \mathbf{1} \right\rangle = 2 \pi \neq 0  \]
holds for any $ u $, $ M_i = M_1 $ holds for any $i $. 
The integral form can be written in 
\[ u_t + \frac{1}{2} \left( u^2 \right)_x + \frac{\gamma}{6} u_x^3 = \ginv{\partial}_x \left( u + \frac{1}{2} u_x^2 \right) + \mathcal{C}(u), \]
where 
\begin{align*}
\mathcal{C} (v) 
&:= \frac{1}{2 \pi} \int_{\Sb} \left( 1 - v_{xx} \right) \left( \frac{1}{2} \left( v^2 \right)_x + \frac{\gamma}{6} v_x^3 - \ginv{\partial}_x \left( v + \frac{1}{2} v_x^2 \right)  \right) \rd x 
= \frac{\gamma}{12 \pi} \int_{\Sb}  v_x^3 \rd x.
\end{align*}
\end{example}

It should be noted that, for such cases as the modified short pulse equation 
(i.e., whether the condition \eqref{eq_condition} holds or not depends on the initial condition), 
$ M_1 $ itself does not describe the whole of the vector field. 
However, if we fix the initial condition $ u_0 $ satisfying the condition, 
they can surely be rewritten in the integral form
by using one implicit constraint $ \mathcal{F} (u(t)) = 0 $. 
This can be done since the procedure in the previous section copes with the single orbit itself, 
whereas the geometric reduction described in this section is to determine the Hilbert manifold consisting of all the orbits, 
and the whole of the vector field on such a manifold.

\subsection{Application: the global well-posedness of the sine-Gordon equation}
\label{subsec_wp_sg}

In this section, in order to illustrate how the proposed procedure work, 
we establish the global well-posedness of the sine-Gordon (sG) equation~\eqref{eq_SG} in light-cone coordinates. 
It should be noted that, to the best of the present authors' knowledge, the well-posedness result can only be found in \cite[Theorem~3.1]{LY2017_2}, which is somewhat limited since it is proved by the local well-posedness of the modified short pulse equation and a reciprocal transformation between the modified short pulse equation and the sine-Gordon equation. 

Prior to the main part, note that, 
in addition to the implicit constraint $ \mathcal{F} (u(t)) = 0 $ ($ \mathcal{F} (v) := \int_{\Sb} \sin v (x) \rd x $), 
the sine-Gordon equation~\eqref{eq_SG} has the conserved quantity: 
\begin{equation}\label{eq_con}
\mathcal{H} (u(t)) = \int_{\Sb} \cos u (t,x) \rd x = \mathcal{H} (u_0) . 
\end{equation}

Here, for the initial value problem of the sine-Gordon equation
\begin{equation}\label{eq_sg}
\begin{cases}
u_{tx} = \sin u \quad & (t>0, x \in \Sb), \\
u(0,x) = u_0 (x) & (x \in \Sb), 
\end{cases}
\end{equation}
we establish the following theorem. 

\begin{theorem}\label{thm_wp_sg}
	Let $ u_0 \in H^1 (\Sb) $ be an initial data satisfying $ \mathcal{F} (u_0) = 0 $ and $ \mathcal{H} (u_0) \neq 0 $. 
	Then, there exists a unique global solution $  u \in C^1 ([0,\infty); H^1 (\Sb)) $ of~\eqref{eq_sg}. 
\end{theorem}

To this end, by using the proposed procedure, we derive the integral form 
\[ u_t = \adps \sin u - \frac{ \left\langle \cos u , \adps \sin u \right\rangle }{ \mathcal{H} (u) } \mathbf{1} . \]
Since $ \mathcal{H} $ is a conserved quantity of the sG equation, 
we consider the following initial value problem 
\begin{equation}\label{eq_sg_int}
\begin{cases}
{\displaystyle u_{t} (t,x) = \left( \adps \sin u (t) \right) (x) - \frac{1}{\mathcal{H}_0} \tilde{\mathcal{C}} (u(t)) \mathbf{1} }\quad & (t>0, x \in \Sb), \\
u(0,x) = u_0 (x) & (x \in \Sb), 
\end{cases}
\end{equation}
where $ \tilde{\mathcal{C}} (v) := \int_{\Sb} \cos v (x) \adps \sin v (x) \rd x $ and $ \mathcal{H}_0 \in \RR $ is a nonzero constant. 
Note that, though $ \mathcal{H}_0 $ will be selected as $ \mathcal{H}_0 = \mathcal{H} (u_0) $ later, 
it is just a general nonzero constant at this moment. 
This reformulation is to make the proof of global well-posedness quite simple; 
as a side effect, Proposition~\ref{prop_ic} (conservation of $ \mathcal{F} $) fails to work anymore, but it can be recovered later (see, Lemma~\ref{lem_sg_int_cl} below). 

It is important to note that, our purpose is the well-posedness of the original sG equation~\eqref{eq_sg}, 
but we consider the initial value problem~\eqref{eq_sg_int} independently for a while, 
and the relation between them will be established after that. 

The following lemma states the global well-posedness of the integral form~\eqref{eq_sg_int} (see Appendix for the proof). 

\begin{lemma}\label{lem_wp_sg_int}
	Let $ u_0 \in H^1 (\Sb) $ be an initial data. 
	Then, there exists a unique global solution $  u \in C^1 ([0,\infty); H^1 (\Sb)) $ of~\eqref{eq_sg_int}. 
\end{lemma}

%

The lemma above reveals the existence of the global solution $u$ of the initial value problem~\eqref{eq_sg_int}. 
In order to elevate it to Theorem~\ref{thm_wp_sg}, we need one more step to overcome the side effect mentioned above. 
The following conservation law will be necessary. 

 \begin{lemma}\label{lem_sg_int_cl}
 	If $u_0 \in H^1 (\Sb) $ satisfies $ \mathcal{F} (u_0) = 0 $ and $ \mathcal{H} (u_0) = \mathcal{H}_0 $, 
 	the global unique solution $u$ in Lemma~\ref{lem_wp_sg_int} satisfies 
 	$ \mathcal{F} (u(t)) = 0 $ and $ \mathcal{H} (u(t)) = \mathcal{H}_0 $. 
 \end{lemma}

\begin{proof}
	By simple calculation, we see
	\begin{align*}
	\frac{\rd}{\rd t} \mathcal{F}(u(t))
	&= \left\langle \cos u (t) , u_t (t) \right\rangle 
	= \frac{\tilde{\mathcal{C}}(u(t))}{\mathcal{H}_0} \left( \mathcal{H}_0 - \mathcal{H} (u(t))  \right), \\
	\frac{\rd}{\rd t} \mathcal{H} (u(t))
	&= \left\langle - \sin u (t) ,  u_t (t) \right\rangle = - \left\langle \sin u(t) , \adps \sin u(t) \right\rangle + \frac{\tilde{\mathcal{C}}(u(t))}{\mathcal{H}_0} \mathcal{F} (u(t)) \\
	&= \frac{\tilde{\mathcal{C}}(u(t))}{\mathcal{H}_0} \mathcal{F} (u(t)). 
	\end{align*}
	Therefore, the assumptions imply
	$ \mathcal{F} (u (t) ) = 0 $ and $ \mathcal{H} (u(t)) = \mathcal{H}_0 $. 
\end{proof}

Now, we are ready to prove Theorem~\ref{thm_wp_sg}. 
When assumptions of Lemma~\ref{lem_sg_int_cl} are satisfied, the solution $u $ of \eqref{eq_sg_int} satisfies 
$ \mathcal{F} (u(t)) = 0$ for any $ t > 0 $, i.e., $ \sin u (t) \in \check{H}^1 (\Sb) $. 
It implies $ \partial_x \adps \sin u (t) = \sin u $ so that $ u $ is also a solution of~\eqref{eq_sg}. 
Therefore, there exists a global solution for the initial value problem~\eqref{eq_sg}. 
The uniqueness is obvious due to the uniqueness of the integral form (Lemma~\ref{lem_wp_sg_int}) 
because all solutions of \eqref{eq_sg} also solve \eqref{eq_sg_int} with appropriate $ \mathcal{H}_0 $. 

\begin{remark}
	Note that, the argument above ceases to work when we deal with $ H^s (\Sb) $ for $ s \ge 2 $, 
	because the map $F : H^s (\Sb) \to H^s (\Sb) $ is not globally Lipschitz in such spaces due to the presence of the sine function (see, e.g., appendix of \cite{BM2008} proving the global well-posedness of the sine-Gordon equation in Euclidean coordinates). 
	Still, by using local Lipschitzness, we can prove the local well-posedness. 
	Since our motivation is not the well-posedness of the sine-Gordon equation itself, 
	we do not step into this issue in the present paper. 
\end{remark}

\subsection{Concluding remarks of Section~\ref{sec_der}}

As shown in the proof of Theorem~\ref{thm_wp_sg}, 
the equivalence of differential and integral forms strongly relies on the common implicit constraint $ \mathcal{F} (u(t) ) = 0 $. 
However, the origins of the constraints are significantly different. 
In the differential form, the implicit constraint is automatically realized by its structure, namely, the property $ \mathbf{1} \in \Range (\partial_x)^{\perp} $ of the spatial differential operator $ \partial_x $. 
On the other hand, in the integral form, 
the implicit constraint is kept as a nontrivial conserved quantity (see, e.g., Proposition~\ref{prop_ic} and Lemma~\ref{lem_sg_int_cl}). 
This difference has a critical impact over discretization, 
which is discussed below. 
Briefly speaking, the former is naturally inherited, 
while the latter is generally lost unless some explicit care is taken such that it is kept.

\section{Discretizations}
\label{sec_class}

Based on the above observation, we consider the finite difference spatial discretization of the initial value problem~\eqref{eq_ivp}. 
For this purpose, we introduce $ u_k : [0,T) \to \RR \ ( k \in \mathbb{Z} )$ as the approximation of $ u ( t , k \Delta x ) $, 
where the discrete periodic boundary conditions $ u_{k+K} = u_{k} $ are imposed, and the spatial mesh size $ \Delta x $ is defined as $ \Delta x = 2 \pi /K $ for some positive integer $K$. 
Since we assume the discrete periodicity, we employ the notation $ u = (u_1,\dots , u_K)^{\top} $. 
Although this is an abuse of symbol, we use this since generally no confusion occurs between this and the continuous solution $ u(t,x) $. 

In \cref{subsec_dis_integral,subsec_dis_diff}, we compare the discretization of the differential and integral forms. 
Then, we classify the existing schemes and derive their new variants in Section~\ref{subsec_class}. 
\Cref{sec_mderivative} is devoted to discuss which discretization is suitable for the mixed derivative. 

Here, we point out an extremely important fact that, 
despite the equivalence in the continuous case, 
discretizations based on the differential and integral forms can be essentially different.

\subsection{Discretization of the integral form}
\label{subsec_dis_integral}

Since numerical methods for linear case~\eqref{eq_linear} had been mainly considered based on the integral form in the literature,  
let us first start with the integral form. 
Although the implicit constraint is kept thanks to its simplicity, 
when it comes to general case~\eqref{eq_nl}, 
the implicit constraint is violated in general as illustrated below. 

Let us for brevity introduce the map $ F : X^{s+k} \to X^s $ such that $ F(u) := \pid f (u) $, 
and consider the discretization 
\begin{equation}\label{eq_dis_integral}
\dot{u}_k + \bar{g}_k (u) = \bar{F}_k (u) + \bar{\mathcal{C}} (u)
\end{equation}
of the integral form~\eqref{eq_nl_int}, where 
$ \bar{g}_k : \RR^K \to \RR $ and $ \bar{F}_k $ are some approximations of $g$ and  $F$, 
and $ \dot{u}_k $ stands for the time derivative of $u_k$.   

In this case, the equation~\eqref{eq_dis_integral} is an ODE, 
and generally no constraint is explicitly accompanied here. 
Thus, unless some special care is taken in the discretization so that 
a discrete counterpart of the implicit constraint $ \mathcal{F} (u(t)) = 0 $ successfully results, 
the solution generally violates the implicit constraint (recall the discussion in the last of the previous section). 
This is in sharp contrast to the continuous case. 

Still, for linear case~\eqref{eq_linear}, i.e., the implicit constraint is linear, 
even when we are based on the integral form, 
we can easily construct a numerical method satisfying a discrete analogue of the linear implicit constraint 
(see, the examples in the following section). 

Another note should go to the fact that, 
again as opposed to the continuous case, 
\eqref{eq_dis_integral} cannot be generally reduced to a differential form in the following sense. 
One may expect that we can obtain a differential form such as
\[ \delta_x \dot{u}_k + \delta_x \bar{g}_k = \delta_x \bar{F}_k (u) \]
by some difference operator $ \delta_x $ 
such as the forward difference $ \fd_x u_k := ( u_{k+1} - u_{k} )/\Delta x $, 
central difference $ \cd_x u_k := ( u_{k+1} - u_{k-1} ) / ( 2 \Delta x ) $, among others.
Unfortunately, however, unless the term $ \delta_x \bar{F}_k (u) $ can be simplified so that no singular operators appear there, 
this implicit DAE is obviously underdetermined.

\subsection{Discretization of the differential form}
\label{subsec_dis_diff}

Next, let us consider the direct discretization of the differential form 
(Miyatake--Yaguchi--Matsuo~\cite{MYM2012} firstly and only introduced the discretization of the differential form). 
We show such a discretization keeps a discrete analogue of the implicit constraint 
so that it can be transformed into another expression, which can be regarded as a discretization of the integral form. 

For simplicity, we consider the discretization in the form
\begin{equation}
\delta_x \left( \dot{u}_k +  \bar{g}_k \left( u \right) \right) = \bar{f}_k \left( u \right), 
\label{eq_direct}
\end{equation}
where $ \bar{f}_k :\RR^K \to \RR $ is some approximation of $f$. 
It should be noted that, the equation~\eqref{eq_direct} is a DAE due to the singularity of $ \delta_x $ (recall Remark~\ref{rem_temp}). 

Here, we introduce the matrix-vector expression 
\begin{equation}
D \left( \dot{u} + \bar{g} (u) \right) = \bar{f} (u),
\end{equation}
where $ D $ is the matrix representation of $ \delta_x $, and $ \bar{g} $ and $ \bar{f} $ are defined as $ \bar{g} (u) := ( \bar{g}_1 (u) ,\dots,\bar{g}_K (u))^{\top} $ and $ \bar{f} (u) := ( \bar{f}_1 (u) ,\dots,\bar{f}_K (u))^{\top} $. 
We assume $ D$ is circulant and $ \mathbf{1}^{\top} D = 0 $, 
where $ \mathbf{1} := (1,\dots,1)^{\top} $. 
These are quite mild assumptions since we impose the discrete periodic boundary condition 
and employ the uniform grid. 
Moreover, we assume $ \rank D = K -1 $ for simplicity. 
 
Then, by multiplying $ \mathbf{1}^{\top} $, we see that  
the solution $u$ of the equation~\eqref{eq_direct} automatically satisfies the implicit constraint 
\begin{equation}\label{eq_dic}
\mathcal{F}_{\rd} (u) := \sum_{k=1}^K \bar{f}_k (u) \Delta x = 0 \qquad (\forall t \in [0,T)).
\end{equation}
Note that, this is a discrete counterpart of the implicit constraint $ \mathcal{F} (u(t)) = 0 $, 
which is a distinct advantage of the differential form. 

Furthermore, discretized differential form can be safely transformed to an integral form, 
which is another advantage. 
To see this, 
let us follow the line of the discussion in Section~\ref{subsec_der_nl}. 
By introducing the Tseng generalized inverse $ \ginv{\delta}_x $ of a difference operator $ \delta_x $, 
the scheme~\eqref{eq_direct} can be transformed into 
\begin{equation}
\dot{u}_k + \bar{g}_k (u) = \ginv{\delta}_x \bar{f}_k \left( u \right) + c(t), 
\label{eq_direct_inv}
\end{equation}
where $c(t)$ does not depend on $k$. 
It should be noted that, the matrix expression of $ \ginv{\delta}_x $ is a generalized inverse matrix of $D$.  
In a way similar to the case of the original PDE (Section~\ref{subsec_der_nl}), 
the implicit constraint enables us to determine $c (t)$ under an assumption as follows. 

Since the value of $\mathcal{F}_{\rd} (u)$ is always $0$ and the solution $u$ of the equation~\eqref{eq_direct_inv} satisfies
\begin{equation}
\frac{\rd}{\rd t} \mathcal{F}_{\rd} (u) = \nabla \mathcal{F}_{\rd} (u) \cdot \dot{u} = \nabla \mathcal{F}_{\rd} (u) \cdot \left( - \bar{g} (u) + \ginv{D} \bar{f} (u) + c(t) \mathbf{1}\right)
\end{equation}
(`$\cdot$' denotes the standard inner product),
the value of $c(t) $ is determined as 
\begin{equation}
c(t) = \mathcal{C}_{\rd} (u) :=  \frac{ \nabla \mathcal{F}_{\rd} (u) \cdot \left( \bar{g} (u) - \ginv{D} \bar{f} (u) \right) }{\nabla \mathcal{F}_{\rd} (u) \cdot \mathbf{1}}
\end{equation} 
for $ t \in (0,T)$ satisfying $ \nabla \mathcal{F}_{\rd} (u) \cdot \mathbf{1} \neq 0 $. 
Thus, in this case, the equation~\eqref{eq_direct} is an implicit DAE with index one 
(when $ \nabla \mathcal{F}_{\rd} (u) \cdot \mathbf{1} = 0 $, we obtain a new constraint and index is more than one). 
Under the assumption $ \nabla \mathcal{F}_{\rd} (u) \cdot \mathbf{1} \neq 0  \ ( t \in (0,T))$ 
(which obviously corresponds to the condition $ \int_{\Sb} \delta \mathcal{F} / \delta u \, \rd x \neq 0 $), 
the equation~\eqref{eq_direct} is equivalent to
\begin{equation}
\dot{u}_k + \bar{g}_k (u) = \ginv{\delta}_x \bar{f}_k (u) + \mathcal{C}_{\rd} (u), 
\label{eq_direct_integral}
\end{equation}
which can be regarded as a discretization of the integral form~\eqref{eq_nl_int}. 

Note that, 
although so far we have considered the simple discretization~\eqref{eq_direct}, 
our strategy can easily be applied to other cases. 
For example, the Ostrovsky equation~\cite{O1978} can be rewritten as
\begin{equation}\label{eq_Ostrovsky}
 u_{tx} + u_x^2 + u u_{xx} + \beta u_{xxxx} = \gamma u,  
 \end{equation}
whose discretization is not necessarily in the form~\eqref{eq_direct_inv}. 
For example, 
one sometimes should employ the spatial discretization which cannot be written in the form~\eqref{eq_direct} in order to maintain the conservation law (see, e.g.,~\eqref{eq_Os_norm_fd}). 
Thus, in general, we can consider 
\begin{equation}\label{eq_direct_general}
\delta_x \dot{u}_k + \left( \overline{\partial_x g} \right)_k (u) = \bar{f}_k (u), 
\end{equation}
where $ (\overline{\partial_x g})_k $ is an approximation of $ \partial_x g $ satisfying $ \sum_{k=1}^K (\overline{\partial_x g})_k (u) = 0 $. 
Even when we deal with such a case, we can similarly derive the corresponding integral form
\begin{align*}
\dot{u}_k + \ginv{\delta}_x \left( \overline{\partial_x g} \right)_k (u) &= \ginv{\delta}_x \bar{f}_k (u) + \frac{ \nabla \mathcal{F}_{\rd} (u) \cdot \left( \ginv{D} \left( \overline{\partial_x g} \right) (u) - \ginv{D} \bar{f} (u) \right) }{\nabla \mathcal{F}_{\rd} (u) \cdot \mathbf{1}}. 
\end{align*}
Although the discretization~\eqref{eq_direct_general} is more general than the simple case~\eqref{eq_direct} and includes some practical numerical methods we show below, 
there is no significant difference between \eqref{eq_direct} and \eqref{eq_direct_general} in view of the transformation into the integral form. 
Thus, for simplicity, we employ the simple case in Theorem~\ref{thm_avtrap}, 
and refer the simple case hereafter.

\subsection{Review of Existing schemes}
\label{subsec_class}

In this section, we classify the existing methods from the viewpoint of Section~\ref{sec_der}, 
and derive their equivalent expressions in another form when possible. 
Although the full-discretizations are defined in the literature, 
we show the corresponding semi discretizations by taking the limit $ \Delta t \to 0 $. 

\begin{table}[htp]
\renewcommand{\arraystretch}{1.2}
\caption{The classification of the existing methods and their equivalent schemes in another form. Schemes in italic are those newly derived in this paper. }
\label{tab_cls}
\centering
\begin{tabular}{c||c|c}
\hline \hline
PDE & Differential form~\eqref{eq_nl} & Integral form~\eqref{eq_nl_int} \\
 & ${\displaystyle u_{tx} + g_x(u) = f(u) }$ & ${\displaystyle u_{t} + g(u) = F(u) + \mathcal{C} (u) } $\\ \hline
Scheme & ${\displaystyle \delta_x \dot{u}_k + \delta_x \bar{g}_k (u) = \bar{f}_k (u)  }$ & $ {\displaystyle \dot{u}_k + \bar{g}_k (u) = \bar{F}_k (u) + \bar{\mathcal{C}} (u) }$ \\ 
 & implicit DAE & ODE \\ \hline \hline
Ostrovsky & \emph{average-difference~\eqref{eq_Os_norm_fd_diff}} & trapezoidal~\eqref{eq_Os_norm_fd} \\
                & \emph{Fourier-spectral~\eqref{eq_Os_norm_ps_diff}} & Fourier-spectral~\eqref{eq_Os_norm_ps} \\ \hline
SG            & average-difference~\eqref{eq_avdiff} & \emph{trapezoidal~\eqref{eq_SG_trap}} \\[2pt] \hline \hline
\end{tabular}
\end{table}

Yaguchi--Matsuo--Sugihara~\cite{YMS2010} introduced the discrete counterpart 
\begin{equation}
\acd u_k = \left( \frac{u_0}{2} + \sum_{i=1}^{k-1} u_i + \frac{u_k}{2} \right) \Delta x - \frac{1}{2 \pi} \sum_{i=1}^K \left( \frac{u_0}{2} + \sum_{j=1}^{i-1} u_j + \frac{u_i}{2} \right) \left( \Delta x \right)^2
\label{eq_inv_fd}
\end{equation}
of $ \check{\partial}_x^{-1} $, and devised the norm-preserving scheme 
\begin{equation}
\dot{u}_k - \frac{1}{3} \left( \cd_x u_k^2 + u_k \cd_x u_k \right) + \beta \cd[3]_x u_k = \gamma \acd u_k 
\label{eq_Os_norm_fd}
\end{equation}
for the Ostrovsky equation~\eqref{eq_Ostrovsky}. 
Note that, $ \acd $ corresponds to the discretization of~\eqref{eq_anti_linear} by the trapezoidal rule.  
Moreover, they devised another norm-preserving scheme
\begin{equation}
\dot{u}_k - \frac{1}{3} \left( \dps u_k^2 + u_k \dps u_k \right) + \beta \dps^3 u_k = \gamma \dps^{\dagger} u_k 
\label{eq_Os_norm_ps}
\end{equation}
by using the Fourier-spectral difference operator $ \dps $ (see, \cite{Fornberg1996} for definition) and its Moore--Penrose pseudoinverse $ \dps^{\dagger} $. 
Note that, as the notation implies, in the finite-dimensional case, 
the maximal Tseng generalized inverse coincides with the Moore--Penrose pseudoinverse (see, \cite[Chapter 9, Theorem~3]{IsraelGreville2003}). 

These schemes are the discretization of the integral form~\eqref{eq_nl_int}, 
which generally cannot be rewritten in the differential form due to the lack of the implicit constraint. 
However, the numerical solutions of them satisfies the constraint $ \sum_{k=1}^K u_k = 0$, 
which is a discrete counterpart of the implicit constraint of the Ostrovsky equation. 
This happens since the (original) implicit constraint is linear. 
Moreover, the discrete counterparts $ \acd $ and $ \dps^{\dagger} $ of $ \ginv{\partial}_x $
can be regarded as a generalized inverse of some difference operators. 
Thus, they have corresponding expression in the differential form~\eqref{eq_ivp}: 
in fact, the trapezoidal scheme~\eqref{eq_inv_fd} can be equivalently rewritten as 
\begin{equation}
\fd_x \left( \dot{u}_k - \frac{1}{3} \left( \cd_x u_k^2 + u_k \cd_x u_k \right) + \beta \cd[3]_x u_k \right) = \gamma \fa_x u_k , 
\label{eq_Os_norm_fd_diff}
\end{equation}
where the forward average operator $ \fa_x $ is defined as $ \fa_x u_k = ( u_{k} + u_{k+1} )/2 $, 
and the Fourier-spectral scheme~\eqref{eq_Os_norm_ps} can be equivalently rewritten as 
\begin{equation}
\dps \left( \dot{u}_k - \frac{1}{3} \left( \dps u_k^2 + u_k \dps u_k \right) + \beta \dps^3 u_k \right) = \gamma u_k. 
\label{eq_Os_norm_ps_diff}
\end{equation}


In the transformation of the schemes~\eqref{eq_Os_norm_ps}, 
the Moore--Penrose pseudoinverse is used as one of the generalized inverses. 
On the other hand, in the transformation of the scheme~\eqref{eq_Os_norm_fd} into \eqref{eq_Os_norm_fd_diff}, 
the summation by the trapezoidal rule is used as one of the generalized inverses of the average-difference $ (\fd_x, \fa_x) $, 
which is recently devised by Furihata--Sato--Matsuo~\cite{FSM2016}.  
It can be generalized as shown in the theorem below. 

\begin{remark}
Strictly speaking, since the average-difference had not been rigorously defined as a linear operator (see, Remark~\ref{rem_avdiff}), 
its generalized inverse can not be defined too at this moment. 
However, as shown in the theorem below, the operator $ \acd $ can be used like as the generalized inverse of the average-difference. 

Moreover, the team including present authors has already obtained some results on this issue (see, Section~\ref{sec_cr}).
In fact, the average-difference can be defined as a linear operator, and 
there actually $ \acd $ can be regarded as its Tseng generalized inverse. 
Due to the restriction of the space and 
since this topic is beyond the scope of this paper, we do not step into such a justification here. 
\label{rem_avdiff_gen} 
\end{remark}

\begin{theorem}
Suppose that the initial condition $ ( u_k (0) \mid k =1,\dots,K) $ satisfies $ \sum_{k=1}^K \bar{f}_k (u(0)) = 0 $. 
Let $u$ be a solution of the initial value problem for the average-difference method 
\begin{equation}
\fd_x \left( \dot{u}_k + \bar{g}_k (u) \right) = \fa_x \bar{f}_k (u) 
\label{eq_avdiff}
\end{equation}
satisfying  $ \sum_{j=1}^K \sum_{k=1}^K \frac{\partial \bar{f}_j}{\partial u_k} (u(t)) \neq 0 \ ( t \in (0,T)) $. 
Then, $u$ is also a solution of 
\begin{equation}
\dot{u}_k + \bar{g}_k (u) = \acd \bar{f}_k (u) + \mathcal{C}_{\rd} (u), 
\label{eq_trap}
\end{equation}
where $\mathcal{C}_{\rd} (u) $ is defined as 
\[ \mathcal{C}_{\rd} (u) :=  \frac{\sum_{j=1}^K \sum_{k=1}^K \frac{\partial \bar{f}_j}{\partial u_k} (u) \left( \bar{g}_k (u) - \acd \bar{f}_k (u) \right)}{\sum_{j=1}^K \sum_{k=1}^K \frac{\partial \bar{f}_j}{\partial u_k} (u) }. \]
Conversely, a solution $u$ of the initial value problem for~\eqref{eq_trap} satisfies~\eqref{eq_avdiff}. 
\label{thm_avtrap}
\end{theorem}

\begin{proof}

First, we derive the integral form~\eqref{eq_trap} from the differential form~\eqref{eq_avdiff}. 
By summing the both sides of \eqref{eq_avdiff} for $ k =1,\dots, j-1$, we see
\begin{align*}
\frac{\dot{u}_j - \dot{u}_1}{\Delta x} + \frac{ \bar{g}_j (u) - \bar{g}_1 (u) }{\Delta x} &= \frac{\bar{f}_0 (u)}{2} + \sum_{k=1}^{j-1} \bar{f}_k (u) + \frac{\bar{f}_j}{2}, 
\end{align*}
which is equivalent to
\begin{equation*}
\dot{u}_k + \bar{g}_k (u) = \acd \bar{f}_k (u) + c(t),
\end{equation*}
where $ c(t)$ does not depend on $k$. 
Note that, the equation above corresponds to~\eqref{eq_inter}. 
Therefore, analogously, in order to determine the value of $ c(t) $, 
we can use the implicit constraint $ \sum_{k=1}^K \fa_x \bar{f}_k (u) = 0 $,
which is satisfied for any solutions of~\eqref{eq_avdiff}. 
Before following the argument~\eqref{eq_key}, 
notice that, 
thanks to the discrete periodicity, $ \sum_{k=1}^K \fa_x \bar{f}_k (u) = \sum_{k=1}^K \bar{f}_k (u) $ holds, 
which allows us to use the simple constraint $ \sum_{k=1}^K \bar{f}_k (u) = 0 $. 
Then, we see
\begin{align*}
0 &= \frac{\rd }{\rd t} \sum_{k=1}^K \bar{f}_k (u) 
= \sum_{j=1}^K \sum_{k=1}^K \frac{ \partial \bar{f}_k }{\partial u_j } (u) \dot{u}_j 
= \sum_{j=1}^K \sum_{k=1}^K \frac{ \partial \bar{f}_k }{\partial u_j } (u) \left( - \bar{g}_k (u) + \acd \bar{f}_j (u) + c(t) \right)
\end{align*}
Thus, we obtain $ c(t) = \mathcal{C}_{\rd} (u) $ under the assumption of the theorem. 

Now, let us prove the converse. 
Note that, for any zero-mean vector $v$ (i.e., $ \sum_{k=1}^K v_k = 0 $ holds), 
$ \fd_x \acd v_k = \fa_x v_k $ is satisfied (this claim can be verified by simple calculation). 
Therefore, the solution $u$ of~\eqref{eq_trap} also satisfies~\eqref{eq_avdiff} due to $ \sum_{k=1}^K \bar{f}_k (u) =0 $, 
which is kept by the definition of $ \mathcal{C}_{\rd} $. 
\end{proof}

By using the theorem above, we see that 
the average-difference method~\cite{SOMF2016+} 
\begin{equation}
\fd_x \dot{u}_k = \fa_x \sin u_k 
\label{eq_SG_avdiff}
\end{equation}
for the sine-Gordon equation~\eqref{eq_SG} is equivalent to 
\begin{equation}\label{eq_SG_trap}
\dot{u}_k = \acd \sin u_k - \frac{ \sum_{k=1}^K \cos u_k \acd \sin u_k }{\sum_{k=1}^K \cos u_k}, 
\end{equation}
unless $ \sum_{k=1}^K \cos u_k = 0 $. 
This assumption, however, can be replaced by that for the initial condition, because $ \sum_{k=1}^K \cos u_k $ is a conserved quantity of the average-difference method~\eqref{eq_SG_avdiff} (see, \cite[Theorem~1]{FSM2016}). 
As can be seen, the integral form~\eqref{eq_SG_trap} has nonlocal operator $ \acd $ and 
thus in this sense \eqref{eq_SG_trap} is more complicated than \eqref{eq_SG_avdiff}. 
Still, the integral form can be used for mathematical analysis: 
the well-posedness of~\eqref{eq_SG_avdiff} can be proved via its integral form by following the discussion in \cref{subsec_wp_sg} 
(strictly speaking, we should employ $ (\fd_x )^{\dagger} \fa_x $ instead of $ \acd $ for such an analysis; see, Remark~\ref{rem_mtg}). 

\vspace{0.5cm}

Summing up all the observations, 
let us close this section with the summary below.  
In general, 
the discretizations of differential and integral forms have the following features, respectively:  
\begin{itemize}
\item Discretizations of the differential form
\begin{itemize}
\item[$+$] are often free from nonlocal operator; 
\item[$+$] automatically have an implicit constraint corresponding to $ \mathcal{F} (u) = 0 $; 
\item[$-$] are implicit DAEs (but usually has index one);
\item[$+$] can be almost always rewritten in the integral form. 
\end{itemize}
\item Discretizations of the integral form
\begin{itemize}
\item[$-$] must have nonlocal operator;
\item[$-$] can lose the implicit constraint; 
\item[$+$] are merely ODEs; 
\item[$-$] cannot be rewritten in the differential form in general. 
\end{itemize}
\end{itemize}

Counting all the above pros and cons, 
we believe that, for actual computation, the discretization of the differential form should be employed. 
For the third points, in particular, 
it should be noted that a reduction to an ODE is known to be unpractical for the large DAE representing an electrical network  involving a large and sparse matrix, 
because the sparsity is destroyed by the reduction (see, e.g.~\cite[Example~9.3]{Ascher-Petzold1998}). 
Since the spatial discretization of the differential form also involves a large and sparse matrix, 
we believe it should be numerically treated as is, namely, without a reduction to an ODE. 

On the other hand, the discretization of the integral form may fit to analyzing the property of the scheme as we will show an example in Section~\ref{sec_mderivative}, 
where we have more intense look at the discretization of the differential form.

\subsection{Discussions on the discretization of the mixed derivative}
\label{sec_mderivative}

In this section, we discuss which difference operators are suitable for the spatial discretization of the mixed derivative. 
Since the method of some mathematical analysis on numerical schemes in the form~\eqref{eq_direct} is 
yet to be investigated as mentioned in Introduction, 
we prefer to be based on another expression~\eqref{eq_direct_integral}, 
which is just an ODE. 
Here, we investigate 
the accuracy of generalized inverse $ \ginv{\delta}_x $ of each difference operator $ \delta_x $. 
This is sufficient because the emergence of $ \ginv{\delta}_x $ is the only distinct property of~\eqref{eq_direct_integral}.  
As a result, we conclude that the average-difference is the best way to discretize the mixed derivative 
among 2nd order differences. 
This consequence agrees very well with the numerical observation for some specific cases~\cite{SOMF2016+,FSM2016}. 
 
Since the difference operator $ \delta_x $ is an approximation of the differential operator $ \partial_x $, 
one may expect that $ \ginv{\delta}_x $ is also an approximation of $ \ginv{\partial}_x $. 
Notice also that it is enough to just consider $ \check{\partial}_x^{-1} $ as $ \ginv{\partial}_x $, 
since $ \ginv{\partial}_x $ is only applied to zero-mean functions in~\eqref{eq_nl_int} 
and $ \ginv{\partial}_x v = \pinv{\partial}_x v $ holds for such functions (recall the note after Definition~\ref{def_mtgi}). 
Thus, we expect the relation
\begin{equation}\label{eq_approx}
\ginv{\delta}_x v_k - \ginv{\delta}_x v_{k-1} \approx \int^{k \Delta x}_{(k-1) \Delta x} v(y) \rd y. 
\end{equation}
for $ v : \Sb \to \RR $ and $ v_k \approx v (k \Delta x) $. 
Here, we assume $ (v_k \mid k = 1,\dots,K ) \in \Range (\delta_x) $, 
since $ \ginv{\delta}_x $ is usually applied to such vectors (see, Section~\ref{subsec_dis_diff}). 
Thanks to $ \Range (\delta_x) \subseteq \Dom (\ginv{\delta}_x ) $, 
this assumption justifies that $ \ginv{\delta}_x$ can be applied to $ v_k $. 
Moreover, again as noted after Definition~\ref{def_mtgi}, 
this assumption implies $ \ginv{\delta}_x v_k = \pinv{\delta}_x v_k $ holds, 
i.e., the values of $ \ginv{\delta}_x v_k $ are the same for any Tseng inverses. 

In order to verify the accuracy of the approximation in \eqref{eq_approx}, 
let us consider $ u^{\omega} (x) = \exp ( \im \omega x) $ ($ \omega \in \mathbb{Z} $), i.e., each frequency component of the Fourier series. 
Then, the exact value $ I_k (\omega) $ of the integration on $[(k-1)\Delta x , k \Delta x]$ can be computed as 
\begin{equation}\label{eq_exact_integral}
I_k (\omega) := \int^{k \Delta x}_{(k-1) \Delta x} \exp ( \im \omega x ) \rd x = \frac{2}{\omega} \exp \left( \im \omega \left( k - \frac{1}{2}\right) \Delta x \right) \sin \frac{\omega \Delta x}{2}. 
\end{equation}

Now, let us consider the approximation $ \bar{I}_k (\omega) := \ginv{\delta}_x u^{\omega}_k - \ginv{\delta}_x u^{\omega}_{k-1} $ of $ I_k $, 
where $ u^{\omega}_k := u^{\omega} (k \Delta x) = \exp ( \im \omega k \Delta x ) $. 
In what follows, we only use the notation $ u^{\omega}_k $ as a single component (i.e., $ u^{\omega}_k $ denotes a scholar, and $ u^{\omega} $ denotes a function itself) in order to avoid the ambiguity of possible confusions between the continuous function and the vectors (in the previous sections). 
Notice that, for any $\omega \in \mathbb{Z}$, the vector $ ( u^{\omega}_k \mid k= 1,\dots,K ) $ is one of the eigenvectors of the matrix representation $D$ of $ \delta_x $, 
because $D$ is assumed to be circulant. 
Namely, $ \delta_x u^{\omega}_k = \lambda_{\omega} u^{\omega}_k $ holds for any $ \omega \in \mathbb{Z}$ and $ k \in \mathbb{Z}$, 
where $ \lambda_{\omega} $ is the corresponding eigenvalue, and $ \lambda_{\omega+K} = \lambda_{\omega} $ holds for any $ \omega \in \mathbb{Z} $. 
Then, 
since $ \ginv{\delta}_x v = \pinv{\delta}_x v$ holds for any $ v \in \Range (\delta_x) $ (recall the discussion immediately after Definition~\ref{def_mtgi}) and the Moore--Penrose pseudoinverse $ \pinv{D} $ of the circulant matrix $D$ 
is also circulant, 
we see
\[ \ginv{\delta}_x u^{\omega}_k = \pinv{\delta}_x u^{\omega}_k = \lambda_{\omega}^{-1} u^{\omega}_k = \lambda_{\omega}^{-1} \exp (\im \omega k \Delta x) \]
for $ \omega \in \{ n \in \mathbb{Z} \mid \lambda_n \neq 0 \} $ 
(recall $ \Dom (\check{\partial}_x^{-1}) = \check{X}^{s-1} = \Range (\partial_x)$, and notice $ ( u_k^{\omega} \mid k=1,\dots,K) \in \Range (\delta_x) \iff \lambda_{\omega} \neq 0 $). 
Therefore, for such $ \omega $, it holds that
\begin{equation}\label{eq_approx_integral}
\bar{I}_k (\omega) := \ginv{\delta}_x u^{\omega}_k - \ginv{\delta}_x u^{\omega}_{k-1} = 2 \im \lambda^{-1}_{\omega} \exp \left( \im \omega \left( k - \frac{1}{2}\right) \Delta x \right) \sin \frac{\omega \Delta x}{2}. 
\end{equation}

By combining~\eqref{eq_exact_integral} and \eqref{eq_approx_integral}, we see $\bar{I}_k (\omega) = \im \lambda_{\omega}^{-1} \omega I_k (\omega) $. 
Furthermore, we can easily compute the relative error $ e (\tilde{\omega} ) $ as follows: 
\begin{equation}\label{eq_re_eig}
e(\tilde{\omega}) := \left| \frac{ \bar{I}_k (\omega) - I_k (\omega) }{ I_k (\omega) } \right| = \left| \im \lambda_{\omega}^{-1} \omega - 1 \right|, 
\end{equation}
where $ \tilde{\omega} = \omega \Delta x $ is the scaled wave number. 

Note that, there are the implicitly defined finite differences such as the average-difference and the compact difference (see, e.g., \cite{L1992}),
 i.e., $ U_x = u $ is discretized as $ \delta_x U_k = \mu_x u_k $ with the pair of a difference operator $\delta_x $ and an average operator $ \mu_x $. 
For example, the average-difference is defined by $ \fd_x U_k = \fa_x u_k $, 
and the compact difference is defined by $ \delta^{a,b,c}_x U_k = \mu^{\alpha,\beta}_x u_k $, where 
\begin{align*}
\delta^{a,b,c}_x U_k &= \frac{2c U_{k+3} + 3b U_{k+2} + 6a U_{k+1} -  6a U_{k-1} - 3b U_{k-2} - 2c U_{k-3}}{12\Delta x},\\
\mu^{\alpha,\beta}_x u_k &= \beta u_{k+2} + \alpha u_{k+1} + u_k + \alpha u_{k-1} + \beta u_{k-2} 
\end{align*}
and $ \alpha , \beta , a, b, c$ are parameters. 
Even in these cases, similar argument can be done by using the eigenvalues of $ \pinv{D} M$ instead of $ \lambda^{-1}_{\omega} $'s, 
where $ D $ and $M$ are the matrix representations of $ \delta_x $ and $ \mu_x$. 

\begin{remark}
The compact difference operators are well-defined thanks to the diagonal dominance of the matrix $M$. 
On the other hand, 
since $M$ is singular if $ K$ is even for the average-difference, 
its definition as a linear operator itself is challenging. 
In this paper, we do not step into this issue as we described in Remark~\ref{rem_avdiff_gen}. 

Therefore, the order of the average-difference cannot be defined in usual sense. 
However, we compare it with 2nd order difference operators since
the average-difference reproduce the exact value up to 2nd order polynomials, 
which is a common feature of 2nd order difference operators (see, e.g., Fornberg~\cite{Fornberg1996}). 
\label{rem_avdiff}
\end{remark}

The relative errors $ e_{\mathrm{CD2}} (\tilde{\omega}) $, $ e_{\mathrm{OD2}} (\tilde{\omega}) $, and $ e_{ \mathrm{AD2} } (\tilde{\omega}) $ of the 2nd order central difference, 2nd order one-sided difference $ ( - u_{k+2} + 4 u_{k+1} - 3 u_{k} )/(2 \Delta x) $, and average-difference can be computed as 
\begin{align*}
e_{\mathrm{CD2}} (\tilde{\omega}) &= \left| \frac{ \tilde{\omega} }{\sin \tilde{\omega} } - 1 \right|,&
e_{\mathrm{OD2}} (\tilde{\omega}) &= \left| \frac{2 \im \tilde{\omega } }{-3+4 \exp (\im \tilde{\omega}) - \exp (2 \im \tilde{\omega})} - 1 \right|\\
e_{\mathrm{AD}} (\tilde{\omega}) &= \left| \frac{\tilde{\omega}}{2 \tan ( \tilde{\omega} / 2 )} - 1 \right| 
\end{align*}
for $ \tilde{\omega} \notin \{ n \pi \mid n \in \mathbb{Z} \}$. 
On the other hand, the relative error $ e_{\mathrm{PS}} (\tilde{\omega}) $ for $ \tilde{\omega} \notin \{ n \pi \mid n \in \mathbb{Z} \}$ of the Fourier-spectral difference is 
\[ e_{\mathrm{PS}} (\tilde{\omega}) = \begin{cases} \left| 2n\pi / ( \tilde{\omega} - 2 n \pi ) \right| & (\tilde{\omega} \in (2n \pi , 2(n+1) \pi ) ) \\ \left| 2 (n+1) \pi / ((2n+1) \pi - \tilde{\omega}) \right| & (\tilde{\omega} \in ( (2n+1) \pi , (2n+2) \pi ) )  \end{cases}. \]
These relative errors are summarized in Figure~\ref{fig_re2}. 

\begin{figure}[htp]
\centering
\begin{tikzpicture}
\begin{axis}[width=9cm,compat = newest,xlabel={Scaled wave number $  \tilde{\omega} = \omega \Delta x$},enlarge x limits=false, 
legend style={at={(1.3,0.7)},anchor=north,legend columns=1},
ylabel ={relative error},ymin=0,ymax=6,xmin=0,xmax=6.284,grid = major,thick]
\addplot[smooth,green] table {data/re255cen2.dat};
\addlegendentry{Central difference}
\addplot[smooth,black] table {data/re255fd2.dat};
\addlegendentry{One-sided difference}
\addplot[smooth,red] table {data/re255ad2.dat};
\addlegendentry{Average-difference}
\addplot[green,densely dotted] table {data/re255ps.dat};
\addlegendentry{Fourier-spectral}
\end{axis}
\end{tikzpicture}
\caption{The relative errors $ e(\tilde{\omega}) $ of 2nd order difference operators and the Frourier-spectral difference operator.}
\label{fig_re2}
\end{figure}
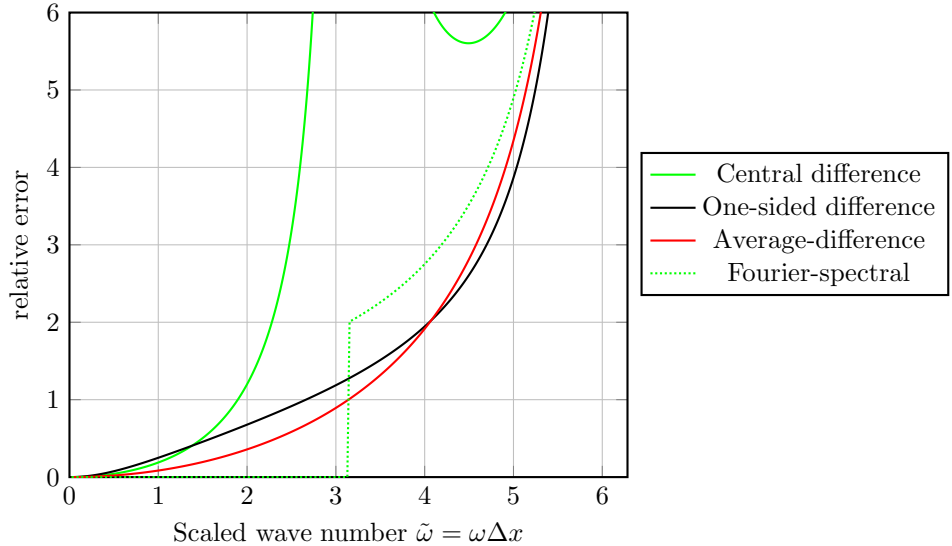

Figure~\ref{fig_re2} shows that 
the average-difference is far better than the central difference, 
in particular, for high frequency components. 
This fact is in good agreement with the observation by Sato--Oguma--Matsuo--Feng~\cite{SOMF2016+} for the sine-Gordon equation (see, \cite[Fig.14 and 16]{SOMF2016+}). 
There, numerical solutions obtained by a finite difference method with the central difference strongly suffer from artificial oscillation, 
while those by the average-difference method reproduce the solution very well. 
Moreover, for frequency components above the Nyquist frequency $ \tilde{\omega} = \pi$, the average-difference is superior to the Fourier-spectral difference. 
This agrees well with the observation by Furihata--Sato--Matsuo~\cite{FSM2016} for the linear Klein--Gordon equation with the square wave (see, \cite[Fig. 3]{FSM2016}). 
There, numerical solutions obtained by the central difference, Fourier-spectral difference, and average-difference are compared, 
and it was found that those by the central and Fourier-spectral differences suffer from artificial oscillation, while again,
the average-difference reproduces better numerical solutions. 

The behavior of the one-sided difference is similar to that of the average-difference method. 
However, the average-difference method is better than the one-sided difference 
for most frequency components. 

As a result, we conclude that the average-difference method is the best among the 2nd order difference methods considered here. 
It should be noted that, 
the target initial value problem in the form~\eqref{eq_ivp} involves 
equations whose solutions tend to have steep fronts (see, e.g., \cite{LPS2009,LPS2010}).  
Therefore, the behavior of high-frequency component is often important, 
and the conclusion here is expected to extend to wide range of PDEs~\eqref{eq_ivp}. 


\begin{remark}
The average-difference method resembles the box scheme. 
In fact, the average-difference method for the advection equation $ u_t = u_x $ 
coincides with the box scheme (detailed analysis on the box scheme and its higher order extensions for the advection equation can be found in Ascher--McLachlan~\cite{AM2004,AM2005}, Frank--Reich~\cite{FR2006} and so on). 
Therefore, the average-difference method also resembles to multisymplectic integrators (see, e.g.,~\cite{Lemkuhler-Reich2004}). 
For example, the Preissman box scheme for the linear Klein--Gordon equation $ u_{tx} = u $ 
coincides with the average-difference method with the implicit midpoint rule for temporal discretization. 

However, the direction of the extension to nonlinear case and their scope are different. 
The average-difference method was investigated in Furihata--Sato--Matsuo~\cite{FSM2016} with special emphasis on its application to the equations in the form $ u_{tx} = \delta \mathcal{H} / \delta u $, 
while the Preissman box scheme is designed for maintaining the discrete multisymplecticity. 
\end{remark}

\section{Concluding Remarks}
\label{sec_cr}

\subsection{Our contributions}
In this paper, we considered PDE-theoretical and numerical treatment of evolutionary equations with a mixed derivative. 

We here emphasize that, the present paper is the first attempt to construct a unified approach for~\eqref{eq_ivp}, 
while currently there are sporadic studies for each specific linear case in the form~\eqref{eq_linear}. 
In the proposed approach, the equivalence of the differential and integral form plays an important role. 
Though several papers dealt with each specific case in the form~\eqref{eq_linear} by following the strategy introduced by Hunter~\cite{H1990}, 
in the present paper, 
we proposed a novel, more unified procedure, which is also applicable to wider class of equations. 
There, we employed the Tseng generalized inverse, 
which is the standard concept of the generalized inverse of the linear operator between Hilbert spaces. 
Moreover, we established the global well-posedness of the sine-Gordon equation by using the newly obtained its integral form. 
It should be noted that, 
although we focused on the periodic domain in this paper, 
we believe that our idea to introduce the Tseng generalized inverse can be applied to other boundary conditions. 

In Section~\ref{sec_class}, we investigated the difference between the discretizations of the differential and integral forms. 
Although most existing numerical methods have been constructed based on their integral forms for the equations in the form~\eqref{eq_linear}, 
we pointed out that the numerical solutions can violate the implicit constraint when one deal with nonlinear case. 
In addition to that, even for the linear case, 
the discretization of the differential form has an advantage that is often free from nonlocal operators. 
Thus, we advocate employing the differential form for actual computation. 
Then, through some mathematical analysis by using the integral form, among several finite difference methods, 
we concluded that the average-difference method is best suited to discretize the mixed derivative. 

\subsection{Future works}

In this paper, 
we left several issues to future works. 

First, the rigorous justification of the infinite dimensional reduction process 
should be done, and we believe that it will be a powerful tool for analyzing 
evolutionary PDEs with a mixed derivative. 

Second, although we focused only on the spatial discretization, 
we should also investigate how to fully discretize them. 
There certainly are a lot of existing works on the temporal discretization of the general implicit DAEs, 
but we think some special treatment will be needed when we deal with DAEs obtained by the spatial discretization of PDEs with the mixed derivative. 

Third, though we used a conserved quantity in order to certify the transformation of the average-difference method for the sine-Gordon equation in Section~\ref{subsec_class}, 
we did not consider any other conservation laws in this paper. 
In view of this, we believe that the combination of the concept of the geometric integration and the framework in this paper should be investigated. 

Finally, detailed analysis and further development including higher order extension of the average-difference method should be discussed. 
As we described in Remarks~\ref{rem_avdiff_gen}, 
we have already obtained some results on this issue and will report it elsewhere soon.

\section*{Acknowledgements}

The authors are indebted to Jason Frank for having the reference~\cite{FR2006} come to the authors' notice. 

\appendix
\section{Proof of Lemma~\ref{lem_wp_sg_int}}

Here, we prove Lemma~\ref{lem_wp_sg_int}. 

	It is sufficient to prove the map $ F $ defined by $ F (v) = \adps \sin v - ( \tilde{\mathcal{C}} (v) / \mathcal{H}_0 ) \mathbf{1}  $ is a map from $ H^1 (\Sb) $ to $ H^1 (\Sb) $ and globally Lipschitz, i.e., there exists a constant $L_F$ such that $ \| F(v) - F(w) \|_1 \le L_F \| v - w \|_1 $ holds for any $ v , w \in H^1 (\Sb) $ ($ \| \cdot \|_s $ stands for the norm with respect to $H^s( \Sb) $). 
	Since the latter condition implies the former one (recall that $ F(0) = 0 $), we only prove the latter. 
	
	First, since the operator norm of $ \adps  $ is equal to one, we see 
	\begin{align*}
	\| \adps \sin v - \adps \sin w \|_0
	&\le  \| \sin v - \sin w \|_0
	\le \| v - w \|_0.
	\end{align*}
	This implies 
	\begin{align*}
	\| \adps \sin v - \adps \sin w \|^2_1 
	&= \| \adps ( \sin v - \sin w )\|^2_0 
	+ \left\| \partial_x \adps \left( \sin v - \sin w \right) \right\|_0^2 \\
	&\le \| v - w \|_0^2 + \left\| \sin v - \sin w \right\|_0^2 
	\le 2 \| v - w \|_0^2 \le 2 \| v - w \|_1^2. 
	\end{align*}
	Then, it holds that
	\begin{align*}
	&\left|  \tilde{\mathcal{C}} (v) - \tilde{\mathcal{C}} (w) \right| \\
	\le{}& \left| \int_{\Sb} \left(  \cos v (x) \adps \sin v (x) -  \cos w (x) \adps \sin w (x) \right) \rd x \right| \\
	\le{}& \left| \int_{\Sb} \left( \frac{\cos v + \cos w}{2} \adps ( \sin v - \sin w ) + ( \cos v  - \cos w ) \frac{\adps (\sin v +  \sin w)}{2} \right) \rd x  \right| \\
	\le{}& \left\| \frac{\cos v + \cos w}{2} \right\|_0 \| \sin v - \sin w \|_0 + \| \cos v - \cos w \|_0 \left\| \frac{\sin v + \sin w}{2} \right\|_0 \\
	\le{}& 2 \sqrt{2\pi} \| v - w \|_{0} \le 2 \sqrt{2\pi} \| v - w \|_1. 
	\end{align*}
	Summing up, we see that 
	\begin{align*}
	\| F(v) - F(w) \|_1 
	\le \sqrt{2} \| v - w \|_1 + \left| \tilde{C} (v) - \tilde{C} (w) \right| \| \mathbf{1} \|_1 \le \left( \sqrt{2} + \frac{4 \pi }{ \mathcal{H}_0} \right) \| v - w \|_1,
	\end{align*}
	which proves the lemma. 

\begin{remark}\label{rem_mtg}
	As shown in the proof above, 
	the maximal Tseng generalized inverse $ \adps $ is extremely useful for such an analysis by the following reasons. 
	First, if we employ the other Tseng generalized inverse $ \pid $, the domain of the  map $ \tilde{F} (v) = \pid \sin v - ( \tilde{\mathcal{C}} / \mathcal{H}_0 ) \mathbf{1} $ is $ \{ v \in H^1(\Sb) \mid \mathcal{F} (v) = 0 \} $ so that the analysis will be complicated. 
	Second, because the operator norms of $ \adps $, $ \partial_x \adps $ and $ \adps \partial_x$ are one, 
	the estimation of each term is usually quite easy (cf. \cite[Lemma~2.6]{LY2017_2} for the evaluation of another similar operator).
\end{remark}


\bibliographystyle{amsplain}
\bibliography{reference}

\end{document}